%% file: main.tex
\documentclass[11pt]{article}

\usepackage{amsmath}
\usepackage{dsfont}
\usepackage{mathrsfs}
\newcommand{\conn}{\mathrel{\longleftrightarrow}}
\input{prefix}

\BibLatexMode{%
   \addbibresource{template.bib}
}

\begin{document}

\title{Capacities and Hitting Times\\ for Simple and (critical) Branching Random Walk Sausages}

\author{Joelle Ishak}

\date{\today}

\maketitle

\begin{abstract}
For a finite set $A\subseteq\mathbb Z^d$, we study the capacities of a sausage obtained by transporting $A$ along the Minkowski sum of simple or critical branching random walks. We first establish
quantitative capacity estimates for the random walk sausage and the branching sausage. This in turn, gives us estimates
for the probability that Minkowski sums of a mixture of walks hit a finite set in
high dimensions. 
Our approach is based on an annular
decomposition of a critical branching random walk into successive waves,
which may have further applications.
\end{abstract}

\section{Introduction}
\label{section-introduction}

In this paper, we study quantitative estimates for the capacities of
random walk and critical branching random walk sausages in high
dimensions. Given a finite set $A\subseteq\mathbb Z^d$ with $0\in A$ and a random
cloud of sites $\mathcal R$, the associated sausage is the Minkowski sum
\begin{equation}
\label{def:sausage}
    \mathcal R+A:=\left\{x+a: x\in\mathcal R,a\in A \right\}.
\end{equation}
Our main objective is to determine, non asymptotically, how the
capacity of $\mathcal R+A$ depends on the size of $\mathcal R$ and on the
\textit{potential-theoretic} size of $A$. We focus in this article on the case where $\mathcal R$ is the range of a simple random walk or of a (critical) branching random walk, as well as Minkowski sums of such ranges.

The study of sausages as in~\eqref{def:sausage} originates in classical works on Brownian motion.
Kesten, Spitzer, and Whitman (see \cite{ItoMcKean1974}, p.~252) showed
that, in dimension $d\geq 3$, the volume of a compact set transported along a Brownian trajectory
grows linearly in time, with asymptotic rate given by its electrostatic
capacity. Spitzer \cite{Spitzer} later established a discrete
counterpart: if $\mathcal S[0,n]$ is the range of a simple
random walk on $\mathbb Z^d$ ($d\geq 3$) up to time $n$, then for every finite set $A\subseteq\mathbb Z^d$,
\begin{equation*}
    \lim_{n\to\infty}\frac{\left|\mathcal S[0,n]+A\right| }{n}=\operatorname{Cap}(A),
\end{equation*}
 where $|\cdot|$ denotes the cardinality of a set and $\operatorname{Cap}(A)$ is the discrete capacity of $A$, defined by
\begin{equation*}
    \operatorname{Cap}(A)
    =\left(\inf\left\{ \mathcal E_g(\nu):\nu \text{ is a probability measure on }A
    \right\}\right)^{-1},
\end{equation*}
where the electrostatic energy reads
\begin{equation*}
    \mathcal E_g(\nu)
    =
    \sum_{x\in A}\sum_{y\in A}g(x,y)\nu(x)\nu(y) ,
\end{equation*}
with $g$ is the Green's function of the random walk,

This link between volume and capacity was recently extended to a large family of capacities by Asselah, Okada, Schapira, and Sousi in \cite{asselah2024branching}. 
For $\gamma>0$, define the $\gamma$-capacity of a finite set
$A\subseteq\mathbb Z^d$ by
\begin{equation}
\label{eq-gamma-capacity-definition}
    \operatorname{Cap}_{\gamma}(A)
    =
    \left(\inf
    \left\{
        \mathcal E_{k_\gamma}(\nu):
        \nu \text{ is a probability measure on }A
    \right\}\right)^{-1},
\end{equation}
where
\begin{equation*}
    \mathcal E_{k_\gamma}(\nu)
    =
    \sum_{x\in A}
    \sum_{y\in A}
    k_\gamma(x,y)\nu(x)\nu(y), \qquad k_\gamma(x,y):=\frac{1}{\left( 1+\|x-y\|\right)^\gamma}.
\end{equation*}
The main result of \cite{asselah2024branching}
shows that, for any $N\geq1$ and $d\geq2N+1$,
\begin{equation}
\label{eq-introduction-aoss-limit}
    \lim_{n\to\infty}
    \frac{
        \operatorname{Cap}_{d-2(N-1)}
        \left(\mathcal S[0,n]+A \right)}{n}
    \asymp
    \operatorname{Cap}_{d-2N}(A).
\end{equation}
Here, $f \asymp g$ means that there exist absolute positive constants $C_1$ and $C_2$ satisfying
$C_1 g \leq f \leq C_2 g$. 
One of our goal here is to
obtain a quantitative, non asymptotic version of
\eqref{eq-introduction-aoss-limit}.

Moreover, we study the corresponding sausage for a critical branching
random walk. A critical branching random walk is defined as follows. We begin with a critical Galton–Watson tree $T_c$ with offspring distribution $\mu$, i.e., an ordered rooted tree in which each node has a random number of children distributed according to $\mu$, with mean one. For each edge, we independently assign a random variable distributed according to $\theta$, a symmetric probability measure on $\mathbb{Z}^d$, which represents the spatial increment associated with that edge. For any $u \in T_c$, we define $V_c(u)$ as  the
sum of the increments along its ancestral line, which we denote $\{v:v \preceq u\}$. The collection $\left(V_c(u)\right)_{u \in T_c}$ is called a critical branching random walk indexed by $T_c$. We write $\mathcal T_c$ for the resulting range. 

\paragraph{Notations.}
Throughout the paper, all random walks are simple symmetric
nearest-neighbour random walks on $\mathbb Z^d$. The critical branching
random walks are indexed by critical Galton--Watson trees with a fixed
offspring distribution $\mu$ of mean one and positive finite variance,
and their spatial increment distribution $\theta$  is the simple symmetric nearest-neighbour measure on $\mathbb Z^d$. Unless stated otherwise, all
random walks, branching trees, and spatial increments appearing
together are mutually independent.

We use $\|\cdot\|$ for the Euclidean norm on $\mathbb Z^d$ and write $B(x,R):=\{y\in\mathbb Z^d:\|y-x\|\leq R\}$ and $B_R :=B(0,R)$.
For a finite nonempty set $A$, let $\operatorname{diam}(A):=\max_{x,y\in A}\|x-y\|$.

\subsection{Main results on capacities}

Let us denote by $\mathcal S(B_R)$ the range of a simple
random walk stopped upon exiting $B_R$, and by $\mathcal T_c(B_R)$ the
range of a critical branching random walk obtained by stopping each
ancestral line when it first exits $B_R$. In other words,
\begin{equation*}
    \mathcal T_c(B_R)
    :=
    \left\{
        u\in\mathcal T_c:
        V_c(v)\in B_R
        \text{ for every }v\preceq u
    \right\}.
\end{equation*}
Our first result give quantitative capacity estimates for both the random walk sausage and the critical branching random walk sausage.

\begin{theorem}
\label{theorem-main-capacity-introduction}
Let $A\subseteq\mathbb Z^d$ be finite and nonempty.

\begin{enumerate}[label=\textup{(\roman*)}]
\item Let $N\geq2$ and $d\geq2N+1$. If $n\geq2\operatorname{diam}(A)^2$,
then
\begin{equation}
\label{eq-introduction-fixed-rw-capacity}
\mathbb E\left[
    \operatorname{Cap}_{d-2(N-1)}
    \left( \mathcal S[0,n]+A \right)
\right]
\asymp
n\operatorname{Cap}_{d-2N}(A).
\end{equation}
Moreover, if $R\geq2\operatorname{diam}(A)$,
then
\begin{equation}
\label{eq-introduction-stopped-rw-capacity}
\mathbb E\left[
    \operatorname{Cap}_{d-2(N-1)}
    \left(
        \mathcal S(B_R)+A
    \right)
\right]
\asymp
R^2\operatorname{Cap}_{d-2N}(A).
\end{equation}

\item Let $N\geq3$ and $d\geq2N+1$. For every sufficiently large $R$ (see Remark~\ref{remark-introduction-brw-scale} below),
\begin{equation}
\label{eq-introduction-stopped-brw-capacity}
\mathbb E\left[
    \operatorname{Cap}_{d-2(N-2)}
    \left(\mathcal T_c(B_R)+A\right)\right]
\asymp
R^2\operatorname{Cap}_{d-2N}(A).
\end{equation}
\end{enumerate}
\end{theorem}

\begin{remark}
\label{remark-introduction-brw-scale}
In part~\textup{(ii)} of
Theorem~\ref{theorem-main-capacity-introduction}, the lower bound
holds under a linear condition of the form $R\geq c N\operatorname{diam}(A)$.
For the upper bound, we require $R$ to satisfy, for $N\geq 3$,
\begin{equation*}
    R\geq C_N\operatorname{diam}(A)^{a_{N-2}},
\end{equation*}
where the exponent $a_m$ is defined in Equation~\eqref{eq-a-m} below and is obtained through an iterative argument in the proof (it corresponds to $a_m= \alpha_{m,1}$ in Equation~\eqref{eq-alpha}).
We note that for $N=3$, $a_1=3$.
\end{remark}

\subsection{Relation with hitting probabilities}

Our approach to Theorem~\ref{theorem-main-capacity-introduction} is
based on the connection between capacities and hitting
probabilities. For a simple random walk  on $\mathbb Z^d$, with $d\geq3$,
it is classical that (see  \cite[Proposition~6.5.1]{modernintro}), for a finite set $A$ containing the origin,
\begin{equation}
\label{eq-introduction-rw-hitting}
 \text{for }\|z\| \geq 2\operatorname{diam}(A), \qquad 
\mathbb P\left(
        z+\mathcal S \conn A
\right)
\asymp
\frac{
    \operatorname{Cap}_{d-2}(A)
}{
    \|z\|^{d-2}
},
\end{equation}
where $\mathcal{S} $ is the infinite time-horizon of the random walk, and for any $A,B\subseteq \mathbb{Z}^d$, we write $A\conn B$ if $A\cap B\neq\emptyset$.

An analogous relation holds for a critical branching random walk on
$\mathbb Z^d$ when $d\geq5$; see \cite[Theorem~1.3]{Zhu1}:
\begin{equation}
\label{eq-introduction-brw-hitting}
 \text{for }\|z\| \geq 2\operatorname{diam}(A), \qquad 
\mathbb P\left( z+\mathcal T_c
   \conn A\right)
\asymp
\frac{
    \operatorname{BCap}(A)
}{
    \|z\|^{d-2}
}.
\end{equation}
Here, $\operatorname{BCap}$ denotes branching capacity; we refer to \cite[Definition~3.1]{Zhu1}
for its definition. 
An important result of \cite[Corollary~1.4]{asselah2023localtimescapacitytransient}
is that the branching capacity is comparable to the $(d-4)$-capacity:
\begin{equation}\label{bcap-equiv-(d-4)cap}
    \operatorname{BCap}(A)
    \asymp
    \operatorname{Cap}_{d-4}(A).
\end{equation}

Asselah, Okada, Schapira, and Sousi \cite[Theorem~1.5]{asselah2024branching} obtained a corresponding
characterization for sums of independent random walk ranges. 
If
$\mathcal S^1,\ldots,\mathcal S^m$ are independent simple random walk
ranges and $d\geq2m+1$, then, for every finite set
$A\subseteq\mathbb Z^d$ containing the origin,
\begin{equation}
\label{eq-introduction-sum-rw-hitting}
\text{for }\|z\|\geq 2\operatorname{diam}(A), \qquad
\mathbb P\left(
        z+\mathcal S^1+\cdots+\mathcal S^m \conn A
\right)
\asymp
\frac{
    \operatorname{Cap}_{d-2m}(A)
}{
    \|z\|^{d-2m}
}.
\end{equation}

Our second main result is to extend AOSS's result \eqref{eq-introduction-sum-rw-hitting} to sums containing both simple random walk and
critical branching random walk ranges.

\begin{theorem}
\label{proposition-main-hitting-introduction}
Let $m\geq0$, $k\geq1$, and $d\geq2m+4k+1$.
Let $\mathcal S^1,\ldots,\mathcal S^m$ and
$\mathcal T_c^1,\ldots,\mathcal T_c^k$ be, respectively, $m$
independent simple random walk ranges and $k$ independent critical
branching random walk ranges, all started from the origin.

Then, for every finite nonempty set
$A\subseteq\mathbb Z^d$ containing the origin and every
$z\in\mathbb Z^d$ sufficiently far from $A$ (see Remark~\ref{remark-introduction-mixed-starting-point}),
\begin{equation*}
\mathbb P\bigg(
        z+\sum_{i=1}^{m}\mathcal S^i
        +\sum_{j=1}^{k}\mathcal T_c^j
    \conn A
\bigg)
\asymp
\frac{
    \operatorname{Cap}_{d-2m-4k}(A)
}{
    \|z\|^{d-2m-2k}
}.
\end{equation*}
\end{theorem}

\begin{remark}
\label{remark-introduction-mixed-starting-point}
In Theorem~\ref{proposition-main-hitting-introduction}, the
statement that $z$ is sufficiently far from $A$ means that
\begin{equation*}
    \|z\|
    \geq
    C_{m,k}\operatorname{diam}(A)^{\alpha_{m,k}},
\end{equation*}
for some explicit exponent $\alpha_{m,k}$ defined by induction, see~\eqref{eq-chi-beta-several-brw}-\eqref{eq-alpha} in the proof.
The precise values of these exponents are not expected to be optimal, as they arise from the successive losses in the induction argument. The constant $C_{m,k}$ may depend on $d$, $m$, $k$, and on
the underlying random walk and offspring distributions, but neither on $A$ nor on $z$.
\end{remark}

The proof of AOSS's result \eqref{eq-introduction-sum-rw-hitting} builds on ideas introduced by Fitzsimmons and Salisbury
\cite{FS1989}, who constructed suitable measures using random times for
additive Markov processes. 
That method
does not apply when the sum contains a critical branching random walk, and we proceed with a different argument.
In particular, one important contribution of the paper is an annular
decomposition of a critical branching random walk into successive
waves. For a simple random walk, annular excursions are naturally
defined through successive stopping times. A branching random walk,
however, may enter and leave the same annulus through many different
ancestral lines, so no single stopping-time decomposition is available.
We replace here stopping times by successive entrance and exit stopping
lines. The main observation is that the corresponding entrance populations have geometrically
decreasing expectations, which makes it possible to sum the
contributions of all waves.
We refer to Section~\ref{section-capacity-sausages} for further details.

Let us mention that 
a recent result of Bai, Chen, and Peres \cite{baiperes} also yields the
comparison in Theorem~\ref{proposition-main-hitting-introduction}.
Their method uses fractal percolation as a canonical model for random
sets whose long-range hitting probabilities are governed by capacity. They first establish a capacity characterization for
Minkowski sums of independent fractal percolations and then derive a
transfer principle for sums of independent random sets. This framework
applies, in particular, to ranges of simple random walks and critical
branching random walks.
Our approach on the other hand is inductive and specifically designed for random walks, simple or branching. 
It is genuinely different, both conceptually and technically, from the approach of \cite{baiperes}. As such it complements our understanding of hitting time for sums of walks. 

\smallskip
We conclude this section by mentioning an interesting corollary of 
Theorem~\ref{theorem-main-capacity-introduction}.
\begin{corollary}
For any $\|z\|\geq 2R$, and any $R \geq C \text{diam}(A)^3$, 
    \begin{equation*}
        \mathbb{P}\left(z+\mathcal{T}_c\conn \mathcal{S}(B_R)+A\right) \asymp \mathbb{P}\left(z+\mathcal{S}\conn \mathcal{T}_c(B_R)+A\right).
    \end{equation*}
\end{corollary}
This means that, at the level of hitting probabilities, the simple random walk and the critical branching random walk can be exchanged: one may stop either process inside a ball while letting the other run over its full range.

\begin{proof}
    In Theorem~\ref{theorem-main-capacity-introduction}, Equations \eqref{eq-introduction-stopped-rw-capacity} and \eqref{eq-introduction-stopped-brw-capacity} with $N=3$ imply that, for $d\geq 7$, any finite nonempty set $A \subset \mathbb{Z}^d$, and any $R \geq C \text{diam}(A)^3$,
    \begin{equation*}
        \mathbb{E}\left[\mathrm{Cap}_{d-4}\left(\mathcal{S}(B_R)+A\right)\right]\asymp \mathbb{E}\left[\mathrm{Cap}_{d-2}\left(\mathcal{T}_c(B_R)+A\right)\right].
    \end{equation*}
    Consequently, combining the relation \eqref{eq-introduction-rw-hitting} between hitting probabilities and capacities, Zhu's \eqref{eq-introduction-brw-hitting} and ASS's  \eqref{bcap-equiv-(d-4)cap}, this concludes the proof.
\end{proof}

\section{Outline of the proofs}
\label{section-outline-proofs}

The proofs of Theorems~\ref{theorem-main-capacity-introduction} and \ref{proposition-main-hitting-introduction} are organized
around a sequence of quantitative capacity and hitting estimates, and
rely on the relation between hitting probabilities and capacities,
as illustrated by Equations
\eqref{eq-introduction-rw-hitting},
\eqref{eq-introduction-brw-hitting}, and
\eqref{eq-introduction-sum-rw-hitting}. 
The proofs will rely on some intermediate
statements that are repeatedly used in the body of the paper.
We now outline the main ideas of the proofs.

\subsection*{Idea of the proof of part~\textup{(i)} of
Theorem~\ref{theorem-main-capacity-introduction}}

The main idea is to decompose the walk into successive pieces
stopped when their displacement first exceeds a prescribed radius. 
The number of pieces completed before time $n$ forms a renewal process. The capacity of
one stopped piece is controlled using AOSS's result
\eqref{eq-introduction-sum-rw-hitting}. Iterating the same argument for several
independent random walks yields the following auxiliary estimate.

\begin{proposition}
\label{prop-capacity-several-stopped-rw}
Let $N\geq2$, let $1\leq k\leq N-1$, and assume that $d\geq2N+1$.
Let $\mathcal S^1,\ldots,\mathcal S^k$ be independent simple random
walk ranges, all started from the origin. For every finite nonempty set
$A\subseteq\mathbb Z^d$ and every
$R\geq2\operatorname{diam}(A)$,
\begin{equation}
\label{eq-upper-several-stopped-rw}
\mathbb E\left[
    \operatorname{Cap}_{d-2(N-k)}
    \left(
        \mathcal S^1(B_R)+\cdots+\mathcal S^k(B_R)+A
    \right)
\right]\lesssim
R^{2k}\operatorname{Cap}_{d-2N}(A).
\end{equation}
If, in addition,
$R\geq2\cdot5^{k-1}\operatorname{diam}(A)$, then
\begin{equation}
\label{eq-lower-several-stopped-rw}
\mathbb E\left[
    \operatorname{Cap}_{d-2(N-k)}
    \left(
        \mathcal S^1(B_R)+\cdots+\mathcal S^k(B_R)+A
    \right)
\right]\gtrsim
R^{2k}\operatorname{Cap}_{d-2N}(A).
\end{equation}
\end{proposition}

\subsection*{Idea of the proof of part~\textup{(ii)} of
Theorem~\ref{theorem-main-capacity-introduction}}

Let us now explain the main idea behind the proof of the branching estimate \eqref{eq-introduction-stopped-brw-capacity}, which deals with a $d-2(N-2)$-capacity of $\mathcal{T}_c(B_R)+A$. By AOSS's result \eqref{eq-introduction-sum-rw-hitting}, the $d-2(N-2)$-capacity of the sausage $\mathcal{T}_c(B_R)+A$ can be estimated through the probability that this sausage is hit by the sum of $(N-2)$ independent random walks. Since $\mathcal{T}_c(B_R) \subset \mathcal{T}_c$, the upper bound in \eqref{eq-introduction-stopped-brw-capacity} can be reduced to the hitting probability of $A$ by the sum of several random walk ranges and one critical branching random walk range. So the next step is to obtain an estimate for this sum.

The main hitting estimate available for the branching random walk is
Zhu's estimate \eqref{eq-introduction-brw-hitting}, which applies when
the target is finite and sufficiently far from the starting point of
$\mathcal{T}_c$. This suggests localizing the random walk ranges and regarding their sum with $A$ as a finite target
for the branching random walk. For the lower bound, this localization
simply consists in stopping each random walk when it exits a prescribed
ball. Zhu's estimate, together with ASS's result  \eqref{bcap-equiv-(d-4)cap}, then yields the
$(d-4)$-capacity of a sum of stopped random walk ranges, which is exactly
the quantity controlled by
Proposition~\ref{prop-capacity-several-stopped-rw}.

For the upper bound, we need to localize the whole random walk range
rather than only a single stopped piece. We first carry out the argument
for one random walk and then iterate it. Following an idea used in
\cite[Proof of Proposition~1.6]{asselah2024branching}
to study intersections between two random sets, we decompose the random
walk range into successive annular traces. Each trace is thus localized
in a bounded region, and the scales are chosen so that the starting point
remains sufficiently far from the corresponding random target for Zhu's
hitting estimate to be applicable. Combining this annular localization
with Proposition~\ref{prop-capacity-several-stopped-rw} yields the following
hitting estimate.



\begin{proposition}
\label{prop-mixed-hitting-probability}
Let $m\geq1$ and $d\geq2m+5$. Let
$\mathcal S^1,\ldots,\mathcal S^m$ be independent simple random walk
ranges and let $\mathcal T_c$ be an independent critical branching
random walk range. If $ \|z\|
    \geq
    C_m\operatorname{diam}(A)^{a_m}$,
where $a_m$ is defined in Equation~\eqref{eq-a-m} below ($a_m:=\alpha_{m,1}$ in Equation \eqref{eq-alpha}), then
\begin{equation*}
\mathbb P\left(
        z+\mathcal S^1+\cdots+\mathcal S^m+\mathcal T_c
   \conn A
\right)\asymp
\frac{
    \operatorname{Cap}_{d-2(m+2)}(A)
}{
    \|z\|^{d-2(m+1)}
}.
\end{equation*}
\end{proposition}

Proposition~\ref{prop-mixed-hitting-probability}  then gives the upper bound in
\eqref{eq-introduction-stopped-brw-capacity}. For the lower bound, the argument proceeds in a different direction. We
compare the stopped branching random walk with the full branching
random walk by Lemma~\ref{lemma-stopped-to-full-brw}. This allows us to apply Zhu's estimate, while the resulting capacity of the
localized random walk ranges is controlled by the lower bound in
Proposition~\ref{prop-capacity-several-stopped-rw}. This completes the
proof of part~\textup{(ii)} of
Theorem~\ref{theorem-main-capacity-introduction}.

\subsection*{Idea of the proof of Theorem~\ref{proposition-main-hitting-introduction}}

Let us first explain the idea in the pure branching case.
The proof is by induction, and the first step is to consider a sum of two critical branching random walks. Here the
annular decomposition used for simple random walks can no longer be
formulated in terms of successive entrance and exit times, since a
branching random walk may hit the same boundary through
many branches. Even though the previous strategy cannot be applied, Zhu's hitting time result \eqref{eq-introduction-brw-hitting} forces us to localize the ranges. We replace stopping times by
\textit{stopping lines} and random walk excursions by \textit{branching waves}.

At a fixed scale, the first stopping line records the pioneer particles entering the annulus, and their descendants are followed
until the corresponding exit line. Iterating this construction produces
successive waves. The key point is that the expected size of the successive entrance lines decreases geometrically. By
criticality, the expected size of the first entrance line is controlled
by the probability that an ordinary random walk reaches the same
boundary. Thus, the total
contribution of all waves is controlled in the same way as the number of
annular excursions in the random walk case.

The capacity contribution of each wave is then estimated by the stopped
branching random walk estimates established in
Theorem~\ref{theorem-main-capacity-introduction}. Combining the number
of entrance particles, the capacity of the corresponding stopped
branching pieces, and the hitting cost of the second branching random
walk yields the case of two branching random walks.

As the number of branching random walks increases, the separation between the starting point and the target set needs to be larger and larger. We therefore further decompose each wave into local branching blocks, with scales chosen so that this separation condition is fulfilled for every block. Conditionally on the corresponding stopping lines, the descendant blocks are independent branching random walks, and the induction hypothesis can be applied to each of them. 
This yields the following intermediate result in the pure branching case.

\begin{proposition}
\label{prop-hitting-several-brw}
Let $p\geq2$ and $d\geq4p+1$. Let
$\mathcal T_c^1,\ldots,\mathcal T_c^p$ be independent critical
branching random walk ranges. If $\|z\|
    \geq
    C_p\operatorname{diam}(A)^{\beta_p}$,
where $\beta_p$ is defined in
Equation~\eqref{eq-chi-beta-several-brw}, then
\begin{equation*}
\mathbb P\left(
        z+\mathcal T_c^1+\cdots+\mathcal T_c^p
    \conn A
\right)
\asymp
\frac{
    \operatorname{Cap}_{d-4p}(A)
}{
    \|z\|^{d-2p}
}.
\end{equation*}
The lower bound remains valid under a linear starting point condition.
\end{proposition}

Finally, starting from Proposition~\ref{prop-hitting-several-brw}, we add
simple random walk ranges one at a time. The same annular localization
argument yields the full mixed estimate in Theorem~\ref{proposition-main-hitting-introduction}.

\subsection*{Organization of the rest of the paper}

\begin{itemize}
\item Section~\ref{section-preliminaries} collects some properties of the $\gamma$-capacity and some probabilistic tools used throughout the paper.
\item Section~\ref{section-capacity-sausages} is devoted to quantitative
capacity estimates for the random walk sausage. 
In particular, we establish the capacity
bounds for stopped random walk ranges and prove part~\textup{(i)} of
Theorem~\ref{theorem-main-capacity-introduction} in Subsection~\ref{Ranges stopped on exiting a ball}.
The proof of Proposition~\ref{prop-capacity-several-stopped-rw}, which deals with sums of stopped random walk
ranges, is presented in Subsection \ref{Sums of stopped random walk ranges}.
\item Section~\ref{section6} focus on quantitative
capacity estimates for the critical
branching random walk sausage.
 In particular, we prove Proposition~\ref{prop-mixed-hitting-probability} about the hitting by the sum of several random walks and one branching in subsections \ref{subsubsection-one-rw-one-brw} and \ref{subsubsection-several-rw-one-brw}.
 We finally prove part~\textup{(ii)} of Theorem~\ref{theorem-main-capacity-introduction} about the capacity of stopped critical branching random walk range in Subsection~\ref{subsubsection-stopped-brw-sausage}.
\item Section~\ref{section8} introduces the annular-wave decomposition and
applies it to the sum of two critical branching random walk ranges, proving Proposition~\ref{prop-hitting-several-brw} in  is the case $p=2$.
\item Section~\ref{section-several-brw} extends the argument to several
critical branching random walks, thereby completing the proof of Proposition~\ref{prop-hitting-several-brw}.
It also establishes the general mixed
hitting estimate of
Theorem~\ref{proposition-main-hitting-introduction}.
\end{itemize}

\section{Preliminaries and auxiliary results}
\label{section-preliminaries}

We collect here the notation and elementary estimates used throughout
the paper. The results quoted from the literature are stated in the
introduction; the present section records the auxiliary facts that are
used repeatedly in the proofs.

\subsection{Properties of discrete capacities}

First of all, let us collect some properties of the $\gamma$-capacity, whose definition is given in~\eqref{eq-gamma-capacity-definition} above.
See Appendix~\ref{appendix-capacities} for the proofs.

\begin{proposition}
\label{prop-basic-capacity-properties}
Let $\gamma>0$, and let $A,B\subseteq\mathbb Z^d$ be finite.
\begin{enumerate}[label=\textup{(\roman*)}]
\item If $A\subseteq B$, then
\begin{equation}
\label{eq-capacity-monotonicity}
    \operatorname{Cap}_{\gamma}(A)
    \leq
    \operatorname{Cap}_{\gamma}(B).
\end{equation}
\item The capacity is subadditive:
\begin{equation}
\label{eq-capacity-subadditivity}
    \operatorname{Cap}_{\gamma}(A\cup B)
    \leq
    \operatorname{Cap}_{\gamma}(A)
    +
    \operatorname{Cap}_{\gamma}(B).
\end{equation}
\item For every $z\in\mathbb Z^d$,
\begin{equation}
\label{eq-capacity-invariances}
    \operatorname{Cap}_{\gamma}(z+A)
    =
    \operatorname{Cap}_{\gamma}(A)
    =
    \operatorname{Cap}_{\gamma}(-A).
\end{equation}
\item One has
\begin{equation}
\label{capofsum}
    \operatorname{Cap}_{\gamma}(A+B)
    \leq
    |B|\operatorname{Cap}_{\gamma}(A).
\end{equation}
\item If $0<\beta\leq\alpha$, then
\begin{equation}
\label{comparisonof2cap}
    \operatorname{Cap}_{\alpha}(A)
    \leq
    \left(
        1+\operatorname{diam}(A)
    \right)^{\alpha-\beta}
    \operatorname{Cap}_{\beta}(A).
\end{equation}
\end{enumerate}
\end{proposition}

\subsection{Critical branching random walks and the many-to-one formula}
Given a tree $\mathcal{T}$, a stopping line $\eta$ is a set of vertices such that no two
distinct vertices of $\eta$ lie on the same ancestral line and, for
every vertex $u$, the event $\{u\in\eta\}$ is determined by the
genealogy and spatial labels along the ancestral line from the root to
$u$. We denote by $\mathcal F_\eta$ the $\sigma$-field generated by the
genealogy and spatial labels revealed up to and including $\eta$.

\begin{lemma}[Branching property at a stopping line]
\label{lemma-branching-property-stopping-line}
Conditionally on $\mathcal F_\eta$, the descendant branching random
walks rooted at the vertices of $\eta$ are independent critical
branching random walks started from their respective spatial positions.
\end{lemma}

\begin{proof}
The offspring variables and spatial increments attached to edges
strictly below distinct vertices of $\eta$ are mutually
independent and are independent of $\mathcal F_\eta$. Their laws are
the same as in the original construction, which proves the claim.
\end{proof}

\begin{lemma}[Many-to-one formula]
\label{lemma-many-to-one-generation}
Assume that the critical branching random walk is started from the
origin. Let $(S_n)_{n\geq0}$ be a random walk with the same jump
distribution, and denote its law by $\mathbf{P}$. Let $F$ be a nonnegative functional of a finite path.
Then, for every $n\geq0$,
\begin{equation}
\label{eq-many-to-one-generation}
\mathbb E\bigg[
    \sum_{|u|=n}
    F\left(
        V_c(u_0),\ldots,V_c(u_n)
    \right)
\bigg]
=
\mathbf E_0\left[
    F(S_0,\ldots,S_n)
\right],
\end{equation}
where $u_0=\varnothing,u_1,\ldots,u_n=u$ is the ancestral line from the root to $u$.
\end{lemma}

\begin{proof}
Conditionally on the genealogical tree, the spatial path along the
ancestral line of each vertex at generation $n$ has the same law as
$(S_0,\ldots,S_n)$. Consequently,
\begin{equation*}
\mathbb E\bigg[
    \sum_{|u|=n}
    F\left(
        V_c(u_0),\ldots,V_c(u_n)
    \right)
    \,\bigg|\,
    \mathcal T_c
\bigg]
=
Z_n\,
\mathbf E_0\left[
    F(S_0,\ldots,S_n)
\right].
\end{equation*}
Taking expectations and using criticality,
$\mathbb E[Z_n]=1$, this proves
\eqref{eq-many-to-one-generation}.
\end{proof}

\section{Capacity estimates for random walk sausages}
\label{section-capacity-sausages}

\subsection{Fixed-time ranges}
\label{sec-fixed-time}
We start by proving Equation \eqref{eq-introduction-fixed-rw-capacity} of Theorem~\ref{theorem-main-capacity-introduction}, \textit{i.e.}\ we consider the range $\mathcal S[0,n]$.

\begin{proof}
    We first prove the lower bound. By Equation~\eqref{eq-introduction-aoss-limit} and the
subadditive ergodic theorem,
\begin{equation*}
\begin{split}
\inf_{m\geq 1}
\frac{
    \mathbb{E}\left[
        \mathrm{Cap}_{d-2(N-1)}
        \left(
            \mathcal{S}[0,m]+A
        \right)
    \right]
}{m}
\asymp
\mathrm{Cap}_{d-2N}(A).
\end{split}
\end{equation*}
Therefore, for every $n\geq 1$,
\begin{equation*}
    \mathbb{E}\left[
        \mathrm{Cap}_{d-2(N-1)}
        \left(
            \mathcal{S}[0,n]+A
        \right)
    \right]
    \gtrsim
    n\,\mathrm{Cap}_{d-2N}(A).
\end{equation*}
This proves the lower bound in Equation~\eqref{eq-introduction-fixed-rw-capacity}.

For the upper bound, let us fix $R>0$. 
We decompose the range of the walk
into pieces stopped upon exiting successive balls of radius $R$. 
Define
$\tau_0:=0$ and, recursively, for every $i\geq 0$,
\begin{equation*}
    \tau_{i+1}
    :=
    \inf\left\{
        m>\tau_i:
        \left\|S_m-S_{\tau_i}\right\|>R
    \right\}.
\end{equation*}
We also define
\begin{equation*}
    \mathcal{N}_n
    :=
    \max\left\{
        i\geq 0:
        \tau_i\leq n
    \right\}.
\end{equation*}
Then we have
\begin{equation*}
    \mathcal{S}[0,n]
    =
    \bigcup_{i=0}^{\mathcal{N}_n}
    \mathcal{S}
    \left[
        \tau_i,
        \tau_{i+1}\wedge n
    \right]\subseteq
    \bigcup_{i=0}^{\mathcal{N}_n}
    \mathcal{S}
    \left[
        \tau_i,
        \tau_{i+1}
    \right].
\end{equation*}
Therefore, by the subadditivity of the
$(d-2(N-1))$-capacity (see Proposition~\ref{prop-basic-capacity-properties}-\textup{(ii)}),
\begin{equation*}
\begin{split}
\mathrm{Cap}_{d-2(N-1)}
\left(
    \mathcal{S}[0,n]+A
\right)
&\leq
\sum_{i=0}^{\mathcal{N}_n}
\mathrm{Cap}_{d-2(N-1)}
\left(
    \mathcal{S}
    \left[
        \tau_i,
        \tau_{i+1}
    \right]
    +A
\right).
\end{split}
\end{equation*}

For every $i\geq 0$, let
\begin{equation*}
    \mathcal{R}_i
    :=
    \left\{
        S_{\tau_i+k}-S_{\tau_i}:
        0\leq k\leq \tau_{i+1}-\tau_i
    \right\}.
\end{equation*}
Since $
    \mathcal{S}
    \left[
        \tau_i,
        \tau_{i+1}
    \right]
    =
    S_{\tau_i}+\mathcal{R}_i$,
the translation invariance of the capacity gives
\begin{equation*}
\begin{split}
\mathrm{Cap}_{d-2(N-1)}
\left(
    \mathcal{S}
    \left[
        \tau_i,
        \tau_{i+1}
    \right]
    +A
\right)
&=
\mathrm{Cap}_{d-2(N-1)}
\left(
    \mathcal{R}_i+A
\right).
\end{split}
\end{equation*}
Consequently,
\begin{equation*}
\begin{split}
\mathrm{Cap}_{d-2(N-1)}
\left(
    \mathcal{S}[0,n]+A
\right)
&\leq
\sum_{i\geq 0}
\mathds{1}_{\{\tau_i\leq n\}}
\mathrm{Cap}_{d-2(N-1)}
\left(
    \mathcal{R}_i+A
\right).
\end{split}
\end{equation*}
By the strong Markov property at time $\tau_i$, conditionally on
$\mathcal{F}_{\tau_i}$, the range $\mathcal{R}_i$ is independent of
$\mathcal{F}_{\tau_i}$ and has the same law as
$\mathcal{S}(B_R)$. Therefore,
\begin{equation*}
\begin{split}
\mathbb{E}\left[
    \mathds{1}_{\{\tau_i\leq n\}}
    \mathrm{Cap}_{d-2(N-1)}
    \left(
        \mathcal{R}_i+A
    \right)
\right]
&=
\mathbb{P}(\tau_i\leq n)
\mathbb{E}\left[
    \mathrm{Cap}_{d-2(N-1)}
    \left(
        \mathcal{S}(B_R)+A
    \right)
\right].
\end{split}
\end{equation*}
Summing over $i\geq 0$, we obtain
\begin{equation}\label{decomp-of-cap-of-S[0,n]}
\begin{split}
\mathbb{E}\left[
    \mathrm{Cap}_{d-2(N-1)}
    \left(
        \mathcal{S}[0,n]+A
    \right)
\right]
&\leq
\mathbb{E}\left[
    \mathcal{N}_n+1
\right]
\mathbb{E}\left[
    \mathrm{Cap}_{d-2(N-1)}
    \left(
        \mathcal{S}(B_R)+A
    \right)
\right].
\end{split}
\end{equation}

Now, let $R= \operatorname{diam}(A)$. Let $\mathcal{S}^1,\ldots,\mathcal{S}^{N-1}$ be
independent simple random walk ranges, independent of
$\mathcal{S}$. Choose $z_R\in\mathbb{Z}^d$ such that $\|z_R\|=10 R$, so that the distance to the set $\mathcal{S}(B_R)+A$ is larger all over than twice its diameter (recall we assume that
$0\in A$). Therefore, conditionally on $\mathcal{S}(B_R)$, we may apply AOSS's Equation~\eqref{eq-introduction-sum-rw-hitting} with $N-1$ random walks. Taking
expectations, we obtain
\begin{equation}
\label{caprwstopped}
\begin{split}
\mathbb{E}\left[
    \mathrm{Cap}_{d-2(N-1)}
    \left(
        \mathcal{S}(B_R)+A
    \right)
\right]
&\lesssim
\|z_R\|^{d-2(N-1)}
\mathbb{P}\left(
        z_R+\mathcal{S}^1+\cdots+\mathcal{S}^{N-1}\conn
        \mathcal{S}(B_R)+A
\right)
\\
&\leq
\|z_R\|^{d-2(N-1)}
\mathbb{P}\left(
        z_R+\mathcal{S}^1+\cdots+\mathcal{S}^{N-1}-\mathcal{S}
    \conn A
\right),
\end{split}
\end{equation}
since $\mathcal{S}(B_R) \subseteq \mathcal{S}$. Applying again AOSS's Equation~\eqref{eq-introduction-sum-rw-hitting} with $N$ independent random walks,
we get
\begin{equation}\label{cap-of-stopped-rw}
\begin{split}
\mathbb{E}\left[
    \mathrm{Cap}_{d-2(N-1)}
    \left(
        \mathcal{S}(B_R)+A
    \right)
\right]
&\lesssim
\|z_R\|^{d-2(N-1)}
\frac{
    \mathrm{Cap}_{d-2N}(A)
}{
    \|z_R\|^{d-2N}
}
\lesssim
R^2\,
\mathrm{Cap}_{d-2N}(A).
\end{split}
\end{equation}

We now estimate the expected number of successive excursions of spatial
diameter $R$ completed by the walk before time $n$. Recall that
$\mathcal{N}_n$ denotes the largest index $i$ such that $\tau_i\leq n$. The classical renewal theorem yields the asymptotic order
$n/R^2$ for each fixed $R$; see
\cite[Section~10.2, Theorem~3]{grimmett-stirzaker}. Our argument,
however, requires a bound that is uniform in both $n$ and $R$, which is
established in the following lemma. We refer to Appendix~\ref{appendix-rw} for its proof.

\begin{lemma}
\label{lemma-number-stopped-pieces}
There exists a constant $C>0$, depending only on the dimension, such
that, for every $R\geq 1$ and every $n\geq 1$,
\begin{equation}
\label{N_n}
    \mathbb{E}\left[
        \mathcal{N}_n+1
    \right]
    \leq
    C\Big(
        1+\frac{n}{R^2}
    \Big).
\end{equation}
In particular, if $n\geq R^2$, then 
$\mathbb{E}[\mathcal{N}_n+1]\lesssim n/ R^2$.
\end{lemma}

Combining Equations \eqref{decomp-of-cap-of-S[0,n]} and \eqref{cap-of-stopped-rw}, with the above lemma, we obtain
\begin{equation*}
\begin{split}
\mathbb{E}\left[
    \mathrm{Cap}_{d-2(N-1)}
    \left(
        \mathcal{S}[0,n]+A
    \right)
\right]
&\lesssim
\left(
    1+\frac{n}{R^2}
\right)
R^2\,
\mathrm{Cap}_{d-2N}(A)
\lesssim
\left(R^2+n
\right)
\mathrm{Cap}_{d-2N}(A).
\end{split}
\end{equation*}

Using the assumption $n\geq 2\operatorname{diam}(A)^2$,
we have $R^2\lesssim n$.
Consequently,
\begin{equation*}
    \mathbb{E}\left[
        \mathrm{Cap}_{d-2(N-1)}
        \left(
            \mathcal{S}[0,n]+A
        \right)
    \right]
    \lesssim
    n\,\mathrm{Cap}_{d-2N}(A).
\end{equation*}
This concludes the proof of the upper bound in Equation~\eqref{eq-introduction-fixed-rw-capacity}.
\end{proof}

\subsection{Ranges stopped on exiting a ball}\label{Ranges stopped on exiting a ball}

We prove now Equation \eqref{eq-introduction-stopped-rw-capacity} of Theorem~\ref{theorem-main-capacity-introduction},
which considers the range $\mathcal S(B_R)$
\begin{proof}
Let $R\geq 2\operatorname{diam}(A)$. The decomposition in Subsection~\ref{sec-fixed-time} above also yields a lower bound for the expected
capacity of the stopped random walk sausage. 
More precisely, let
$R\geq 2\operatorname{diam}(A)$ and $n\geq R^2$. By the range
decomposition and Lemma~\ref{lemma-number-stopped-pieces},
\begin{equation*}
\begin{split}
\mathbb{E}\left[
    \mathrm{Cap}_{d-2(N-1)}
    \left(
        \mathcal{S}[0,n]+A
    \right)
\right]
&\leq
\mathbb{E}\left[
    \mathcal{N}_n+1
\right]
\mathbb{E}\left[
    \mathrm{Cap}_{d-2(N-1)}
    \left(
        \mathcal{S}(B_R)+A
    \right)
\right]
\\
&\lesssim
\frac{n}{R^2}
\mathbb{E}\left[
    \mathrm{Cap}_{d-2(N-1)}
    \left(
        \mathcal{S}(B_R)+A
    \right)
\right].
\end{split}
\end{equation*}
Consequently,
\begin{equation}
\label{forlowerboundcapofrw}
\begin{split}
\mathbb{E}\left[
    \mathrm{Cap}_{d-2(N-1)}
    \left(
        \mathcal{S}(B_R)+A
    \right)
\right]
&\gtrsim
R^2
\frac{
    \mathbb{E}\left[
        \mathrm{Cap}_{d-2(N-1)}
        \left(
            \mathcal{S}[0,n]+A
        \right)
    \right]
}{n}
\\
&\geq
R^2
\inf_{m\geq 1}
\frac{
    \mathbb{E}\left[
        \mathrm{Cap}_{d-2(N-1)}
        \left(
            \mathcal{S}[0,m]+A
        \right)
    \right]
}{m}
\gtrsim
R^2\,
\mathrm{Cap}_{d-2N}(A),
\end{split}
\end{equation}
where the last inequality follows from Equation~\eqref{eq-introduction-aoss-limit}. So we get the desired lower bound.

On the other hand, Equation~\eqref{cap-of-stopped-rw} directly gives
\begin{equation*}
\begin{split}
\mathbb{E}\left[
    \mathrm{Cap}_{d-2(N-1)}
    \left(
        \mathcal{S}(B_R)+A
    \right)
\right]
&\lesssim
R^2\,\mathrm{Cap}_{d-2N}(A).
\end{split}
\end{equation*}
Combining the two inequalities proves Equation~\eqref{eq-introduction-stopped-rw-capacity}.
\end{proof}
\subsection{Sums of stopped random walk ranges}\label{Sums of stopped random walk ranges}
We can now generalize the result above and prove Proposition~\ref{prop-capacity-several-stopped-rw}
which deals with the sum of several random walk ranges.

\begin{proof}
We first prove the upper bound \eqref{eq-upper-several-stopped-rw}. Set
\begin{equation*}
    K_R
    :=
    \mathcal{S}^1(B_R)+\cdots+
    \mathcal{S}^k(B_R)+A.
\end{equation*}
Let
$\widetilde{\mathcal{S}}^{k+1},\ldots,
\widetilde{\mathcal{S}}^{N}$
be $N-k$ independent simple random walk ranges, independent of
$K_R$. Choose $z_R\in\mathbb{Z}^d$ such that $\|z_R\|=c_k R$, where $c_k>0$ is sufficiently large. Since $\operatorname{diam}(K_R)\lesssim \|z_R\|$,
we may apply AOSS's equation~\eqref{eq-introduction-sum-rw-hitting} and obtain
\begin{equation*}
\begin{split}
\mathbb{E}\left[
    \mathrm{Cap}_{d-2(N-k)}
    \left(
        K_R
    \right)
\right]
&\lesssim
\|z_R\|^{d-2(N-k)}
\mathbb{P}\left(
        z_R+
        \widetilde{\mathcal{S}}^{k+1}
        +\cdots+
        \widetilde{\mathcal{S}}^{N}
    \conn K_R
\right)\\
& \leq \|z_R\|^{d-2(N-k)}
\mathbb{P}\left(
        z_R+
        \widetilde{\mathcal{S}}^{k+1}
        +\cdots+
        \widetilde{\mathcal{S}}^{N}-\mathcal{S}^1-\cdots-\mathcal{S}^k
   \conn A
\right).\\
\end{split}
\end{equation*}

Applying again AOSS's Equation~\eqref{eq-introduction-sum-rw-hitting} with $N$ random walks yields
\begin{equation*}
\begin{split}
\mathbb{E}\left[
    \mathrm{Cap}_{d-2(N-k)}
    \left(
        K_R
    \right)
\right]
&\lesssim
\|z_R\|^{d-2(N-k)}
\frac{
    \mathrm{Cap}_{d-2N}(A)
}{
    \|z_R\|^{d-2N}
}
\lesssim
R^{2k}\,
\mathrm{Cap}_{d-2N}(A).
\end{split}
\end{equation*}
This proves \eqref{eq-upper-several-stopped-rw}.
We now prove the lower bound \eqref{eq-lower-several-stopped-rw} by induction on $k$.
For $k=1$, the result is precisely the lower bound in
Equation~\eqref{eq-introduction-stopped-rw-capacity}.

Assume that the result holds for some $k\geq1$, simultaneously for all
admissible values of $N$, and consider $k+1\leq N-1$. Suppose that $R\geq
    2\cdot5^k\operatorname{diam}(A)$, and set $R_1:=\frac{R}{5^k}$.
In particular, $R_1\geq2\operatorname{diam}(A)$. By the monotonicity of capacity,
\begin{equation*}
\begin{split}
&\mathbb{E}\left[
    \mathrm{Cap}_{d-2(N-(k+1))}
    \left(
        \mathcal{S}^1(B_R)+\cdots+
        \mathcal{S}^{k+1}(B_R)+A
    \right)
\right]
\\
&\qquad\qquad\geq
\mathbb{E}\left[
    \mathrm{Cap}_{d-2(N-(k+1))}
    \left(
        \mathcal{S}^1(B_R)+\cdots+
        \mathcal{S}^{k}(B_R)
        +\mathcal{S}^{k+1}(B_{R_1})+A
    \right)
\right].
\end{split}
\end{equation*}
Set $ A_1
    :=
    \mathcal{S}^{k+1}(B_{R_1})+A$. Since $\operatorname{diam}(A_1)
    \leq
    2R_1+\operatorname{diam}(A)$,
our assumptions imply
\begin{equation*}
\begin{split}
2\cdot5^{k-1}\operatorname{diam}(A_1)
&\leq
2\cdot5^{k-1}
\left(
    2R_1+\operatorname{diam}(A)
\right)
\\
&=
\frac{4R}{5}
+
2\cdot5^{k-1}\operatorname{diam}(A)
\leq
R.
\end{split}
\end{equation*}
Conditionally on $A_1$, we may therefore apply the induction
hypothesis with \(N\) replaced by \(N-1\). This gives
\begin{equation*}
\begin{split}
\mathbb{E}\left[
    \mathrm{Cap}_{d-2(N-(k+1))}
    \left(
        \mathcal{S}^1(B_R)+\cdots+
        \mathcal{S}^{k}(B_R)+A_1
    \right)
    \,\middle|\,
    A_1
\right]
&\gtrsim
R^{2k}\,
\mathrm{Cap}_{d-2(N-1)}(A_1).
\end{split}
\end{equation*}
Taking expectations and applying the lower bound of AOSS's Equation~\eqref{eq-introduction-stopped-rw-capacity} at radius $R_1$, we obtain
\begin{equation*}
\begin{split}
&\mathbb{E}\left[
    \mathrm{Cap}_{d-2(N-(k+1))}
    \left(
        \mathcal{S}^1(B_R)+\cdots+
        \mathcal{S}^{k+1}(B_R)+A
    \right)
\right]
\\
&\qquad\qquad\gtrsim
R^{2k}
\mathbb{E}\left[
    \mathrm{Cap}_{d-2(N-1)}
    \left(
        \mathcal{S}^{k+1}(B_{R_1})+A
    \right)
\right]
\\
&\qquad\qquad\gtrsim
R^{2k}R_1^2\,
\mathrm{Cap}_{d-2N}(A)\gtrsim
R^{2(k+1)}
\mathrm{Cap}_{d-2N}(A).
\end{split}
\end{equation*}
This completes the induction and proves the lower bound \eqref{eq-lower-several-stopped-rw}.
\end{proof}

\section{Capacity estimates for branching random walk sausages}
\label{section6}

In this section, we prove part~\textup{(ii)} of
Theorem~\ref{theorem-main-capacity-introduction}. We first establish the
mixed hitting estimate stated in
Proposition~\ref{prop-mixed-hitting-probability}. We then combine this estimate with Proposition~\ref{prop-capacity-several-stopped-rw} and the additional Lemma~\ref{lemma-stopped-to-full-brw}, stated below, to derive the upper and lower bounds for the capacity of the stopped critical branching random walk sausage.

\subsection{Hitting by the sum of one random walk and one branching random walk}
\label{subsubsection-one-rw-one-brw}
We start by proving the case $m=1$ of Proposition~\ref{prop-mixed-hitting-probability}. Let us set $a_1=\alpha_{1,1}=3$. 

\begin{proof}
We first prove the lower bound.
For any $R \geq 1$, since $\mathcal{S}(B_R)\subseteq\mathcal{S}$, we have
\begin{equation*}
\begin{split}
\mathbb{P}\left(z+\mathcal{S}+\mathcal{T}_c
    \conn A
\right)
&\geq
\mathbb{P}\left(
        z+\mathcal{S}(B_R)+\mathcal{T}_c
    \conn A
\right)
\\
&=
\mathbb{P}\left(
        z+\mathcal{T}_c
    \conn A-\mathcal{S}(B_R)
\right).
\end{split}
\end{equation*}

We choose $R\asymp \|z\|$ such that $
    \|z\|
    \geq
    2\operatorname{diam}\left(A-\mathcal{S}(B_R)\right)$.
Therefore, conditionally on $\mathcal{S}(B_R)$, Zhu's estimate~\eqref{eq-introduction-brw-hitting}
together with ASS's Equation \eqref{bcap-equiv-(d-4)cap} give
\begin{equation*}
\begin{split}
\mathbb{P}\left(
        z+\mathcal{T}_c
    \conn A-\mathcal{S}(B_R)\right)
&\gtrsim
\frac{
    \mathbb E\left[\mathrm{Cap}_{d-4}\left(A-\mathcal{S}(B_R)\right)
\right]}{
    \|z\|^{d-2}}.
\end{split}
\end{equation*}
Applying the lower bound in Equation~\eqref{eq-introduction-stopped-rw-capacity} of Theorem \ref{theorem-main-capacity-introduction} with $N=3$, we
get
\begin{equation*}
\begin{split}
\mathbb{E}\left[
    \mathrm{Cap}_{d-4}\left(A-\mathcal{S}(B_R)\right)\right]
&\gtrsim
R^2\,\mathrm{Cap}_{d-6}(A).
\end{split}
\end{equation*}
Since $R\asymp\|z\|$, it follows that
\begin{equation*}
\begin{split}
\mathbb{P}\left(
        z+\mathcal{S}+\mathcal{T}_c
   \conn A
\right)
&\gtrsim
\frac{
    R^2\mathrm{Cap}_{d-6}(A)
}{
    \|z\|^{d-2}
}
\gtrsim
\frac{
    \mathrm{Cap}_{d-6}(A)
}{
    \|z\|^{d-4}
}.
\end{split}
\end{equation*}
This proves the lower bound.

We now prove the upper bound. Assume that $\|z\| \geq 2 \text{diam}(A)^3$. By reflection symmetry,
$-\mathcal{S}$ has the same law as $\mathcal{S}$. Therefore,
\begin{equation*}
\begin{split}
\mathbb{P}\left(
    z+\mathcal{S}+\mathcal{T}_c
    \conn A
\right)
&=
\mathbb{P}\left(
        z+\mathcal{T}_c\conn
        \mathcal{S}+A
\right).
\end{split}
\end{equation*}
We decompose the range
$\mathcal{S}$ into annular excursions. This decomposition is adapted from the
annular decomposition used in \cite{asselah2024branching}.
For every $i\in\mathbb{Z}$, define $r_i:=2^i\|z\|$,
and let $I$ be the maximal integer such that $r_{-I}\geq 4 \operatorname{diam}(A)$.
For every $i\geq-I$, set
\begin{equation*}
    \mathscr{C}_i
    :=
    B(z,r_{i+1})\setminus B(z,r_{i-1}) \qquad  \Sigma_i:=\partial B(z,r_i).
\end{equation*}
For every $k\geq0$, define the starting time of the $k$-th excursion at level~$i$ by
\begin{equation*}
    \tau_i^k
    :=
    \inf\left\{
        n>\sigma_i^{k-1}:
        S_n\in\Sigma_i
    \right\}
\end{equation*}
with by convention $\sigma_i^{-1}:=-1$ and 
\begin{equation*}
    \sigma_i^k
    :=
    \inf\left\{
        n\geq\tau_i^k:
        S_n\in\Sigma_{i-1}\cup \Sigma_{i+1}
    \right\} 
\end{equation*}
is the ending time of the $k$-th excursion.
As usual, we use the convention that the infimum of the empty set is
infinite. 
Define the set of all excursions at the $i$-th level
\begin{equation*}
    \mathcal{S}^{(i)}
    :=
    \bigcup_{k:\,\tau_i^k<\infty}
    \mathcal{S}[\tau_i^k,\sigma_i^k] .
\end{equation*}
We refer to Figure~\ref{fig-excursions} for an illustration.

\begin{figure}[tb]
    \centering
    \includegraphics[width=0.55\linewidth]{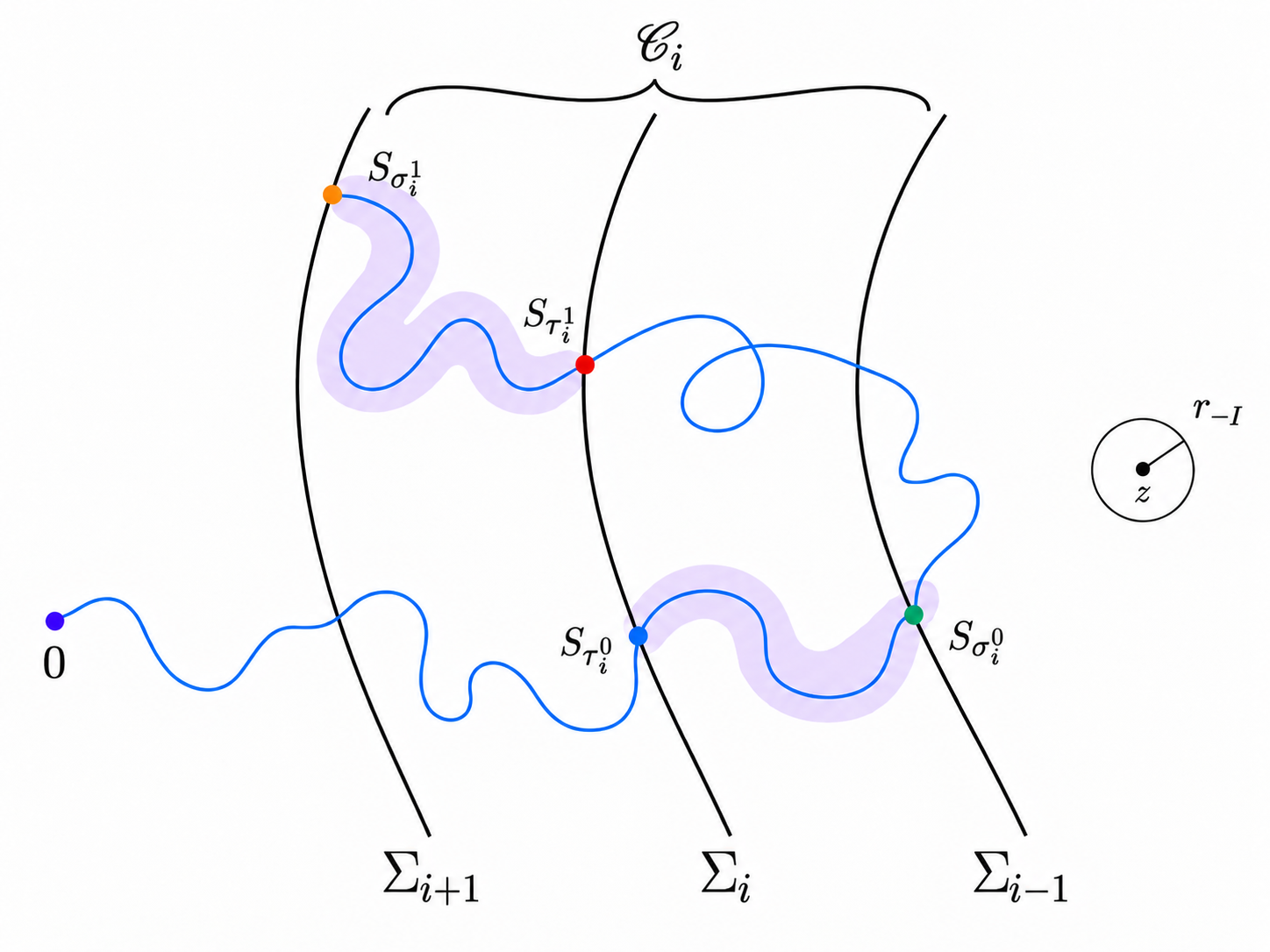}
    \caption{The first two excursions at the $i$-th level.}
    \label{fig-excursions}
\end{figure}

This annular decomposition gives
\begin{equation*}
    \mathcal{S}
    \subseteq
    \left(
        \mathcal{S}\cap B(z,r_{-I})
    \right)
    \cup
    \bigcup_{i\geq-I}\mathcal{S}^{(i)}.
\end{equation*}
Consequently, by a union bound,
\begin{equation}
\label{eq-decomposition-upper-mixed}
\begin{split}
\mathbb{P}\left(
        z+\mathcal{S}+\mathcal{T}_c
    \conn A
\right)
&\leq
\mathbb{P}\left(
        z+\mathcal{T}_c
   \conn
            \mathcal{S}\cap B(z,r_{-I})+A
\right)+
\sum_{i\geq-I}
\mathbb{P}\Big(
        z+\mathcal{T}_c\conn
        \mathcal{S}^{(i)}+A
\Big).
\end{split}
\end{equation}

We first estimate the first term in
\eqref{eq-decomposition-upper-mixed}. Since $0\in A$ and
$r_{-I}\geq4\text{diam}(A)$, we have
\begin{equation*}
\begin{split}
\mathbb{P}\left( z+\mathcal{T}_c
    \conn
            \mathcal{S}\cap B(z,r_{-I})+A
\right)
&\leq
\mathbb{P}\left(
    \mathcal{S}
    \conn
        \left(
            z+\mathcal{T}_c
        \right)
        \cap B(z,2r_{-I})
        -A
\right).
\end{split}
\end{equation*}
Since $\|z\| \geq2\operatorname{diam}(A)^3$, we can apply Equation~\eqref{eq-introduction-rw-hitting} and get
\begin{equation*}
\begin{split}
\mathbb{P}\left(
    \mathcal{S}\conn 
            \left(
            z+\mathcal{T}_c
        \right)
        \cap B(z,2r_{-I})
        -A
\right)& \lesssim\frac{\mathbb{E}\left[\mathrm{Cap}_{d-2}\left(\left(
            z+\mathcal{T}_c
        \right)
        \cap B(z,2r_{-I})
        -A\right)\right]}{\|z\|^{d-2}}\\
&
\lesssim
\frac{
    \mathrm{Cap}_{d-2}(A)
}{
    \|z\|^{d-2}
}
\mathbb{E}\left[
    \left|
        \left(
            z+\mathcal{T}_c
        \right)
        \cap B(z,2r_{-I})
    \right|
\right],
\end{split}
\end{equation*}
where in the last inequality, we use \eqref{capofsum} and \eqref{eq-capacity-invariances}. The expected number of distinct sites visited by the critical
branching random walk is bounded by its expected occupation measure.
Therefore,
\begin{equation*}
\begin{split}
\mathbb{E}\left[
    \left|
        \left(
            z+\mathcal{T}_c
        \right)
        \cap B(z,2r_{-I})
    \right|
\right]
&\leq
\sum_{x\in B(z,2r_{-I})}g(z,x).
\end{split}
\end{equation*}
Comparing the sum with an integral, we get
\begin{equation*}
    \sum_{x\in B(z,2r_{-I})}g(z,x)\lesssim\int_{0}^{2r_{-I}}r^{d-1}r^{2-d}\,dr\lesssim r_{-I}^2.
\end{equation*}
Using Equation~\eqref{comparisonof2cap}, $\mathrm{Cap}_{d-2}(A)
    \lesssim
   \operatorname{diam}(A)^4\,\mathrm{Cap}_{d-6}(A)$.
It follows that
\begin{equation*}
\begin{split}
\mathbb{P}\left(
        z+\mathcal{T}_c
    \conn \mathcal{S}\cap B(z,r_{-I})+A
\right)
&\lesssim
\frac{
    \operatorname{diam}(A)^6\mathrm{Cap}_{d-6}(A)
}{
    \|z\|^{d-2}
}.
\end{split}
\end{equation*}
Finally, since $\|z\|\geq2\operatorname{diam}(A)^3$,
\begin{equation}
\label{eq-inner-annular-bound}
\begin{split}
\mathbb{P}\left(
        z+\mathcal{T}_c
   \conn
            \mathcal{S}\cap B(z,r_{-I})+A
\right)
&\lesssim
\frac{
    \mathrm{Cap}_{d-6}(A)
}{
    \|z\|^{d-4}
}.
\end{split}
\end{equation}

We now estimate the second term in
\eqref{eq-decomposition-upper-mixed}. Since $r_i\geq r_{-I}\geq4\operatorname{diam}(A)$,
the set $\mathcal{S}^{(i)}+A$ is contained in
$B(z,\frac{9r_i}{4}) \setminus B(z,\frac{r_i}{4})$.
This annulus can be covered by a number, depending only on the
dimension, of sets of diameter at most $r_i/16$. Applying Zhu's Equation~\eqref{eq-introduction-brw-hitting} to each of these sets, together with ASS's Equation \eqref{bcap-equiv-(d-4)cap}, and using the monotonicity of capacity, we obtain
\begin{equation}
\label{eq-annular-hitting-capacity}
\begin{split}
\mathbb{P}\left(
        z+\mathcal{T}_c
   \conn
        \mathcal{S}^{(i)}+A
\right)
&\lesssim
\frac{
    \mathbb{E}\left[
        \mathrm{Cap}_{d-4}
        \left(
            \mathcal{S}^{(i)}+A
        \right)
    \right]
}{
    r_i^{d-2}
}.
\end{split}
\end{equation}
By the subadditivity of capacity and the strong Markov property,
\begin{equation*}
\begin{split}
\mathbb{E}\left[
    \mathrm{Cap}_{d-4}
    \left(
        \mathcal{S}^{(i)}+A
    \right)
\right]
&\lesssim
p_i
\sup_{x\in\Sigma_i}
    \sum_{k\geq0}
    \mathbb{P}_x\left(
        \tau_i^k<\infty
    \right)
\sup_{x\in\Sigma_i}
\mathbb{E}_x\left[
    \mathrm{Cap}_{d-4}
    \left(
        \mathcal{S}[0,\sigma_i^0]+A
    \right)
\right],
\end{split}
\end{equation*}
where $
    p_i:=\mathbb{P}\left(\tau_i^0<\infty\right)$.
By transience of the walk,
\begin{equation*}
    \sup_{i\geq-I}
    \sup_{x\in\Sigma_i}
    \sum_{k\geq0}
    \mathbb{P}_x\left(
        \tau_i^k<\infty
    \right)
    \lesssim1.
\end{equation*}
Furthermore, $p_i \lesssim (r_i/\|z\|)^{d-2}$ when $-I\leq i <0$ and $p_i \lesssim 1$ when $i\geq 0$.

For every $x\in\Sigma_i$, the excursion
$\mathcal{S}[0,\sigma_i^0]$ is contained in the range of a walk
started at $x$ and stopped upon exiting $B(x,3r_i)$. Thus,
\begin{equation*}
\begin{split}
\mathbb{E}_x\left[
    \mathrm{Cap}_{d-4}
    \left(
        \mathcal{S}[0,\sigma_i^0]+A
    \right)
\right]
&\leq
\mathbb{E}\left[
    \mathrm{Cap}_{d-4}
    \left(
        \mathcal{S}(B_{3r_i})+A
    \right)
\right].
\end{split}
\end{equation*}
Since $3r_i\geq2\operatorname{diam}(A)$, the upper bound in
Equation~\eqref{eq-introduction-stopped-rw-capacity} of Theorem \ref{theorem-main-capacity-introduction}, applied with $N=3$, yields
\begin{equation*}
\begin{split}
\mathbb{E}\left[
    \mathrm{Cap}_{d-4}
    \left(
        \mathcal{S}(B_{3r_i})+A
    \right)
\right]
\lesssim
r_i^2\,\mathrm{Cap}_{d-6}(A).
\end{split}
\end{equation*}
Consequently,
\begin{equation}
\label{eq-capacity-annular-piece}
\begin{split}
\mathbb{E}\left[
    \mathrm{Cap}_{d-4}
    \left(
        \mathcal{S}^{(i)}+A
    \right)
\right]
&\lesssim
\begin{cases}
    \displaystyle
    \|z\|^{2-d}r_i^d\,
    \mathrm{Cap}_{d-6}(A),
    &-I\leq i<0,\\[1.5ex]
    \displaystyle
    r_i^2\,
    \mathrm{Cap}_{d-6}(A),
    &i\geq0.
\end{cases}
\end{split}
\end{equation}
Combining \eqref{eq-annular-hitting-capacity} and
\eqref{eq-capacity-annular-piece}, we obtain
\begin{equation*}
\begin{split}
\sum_{i=-I}^{-1}
\mathbb{P}\left(
        z+\mathcal{T}_c
    \conn
        \mathcal{S}^{(i)}+A
\right)
&\lesssim
\mathrm{Cap}_{d-6}(A)
\|z\|^{2-d}
\sum_{i=-I}^{-1}r_i^2
\lesssim
\frac{
    \mathrm{Cap}_{d-6}(A)
}{
    \|z\|^{d-4}
}.
\end{split}
\end{equation*}
For the remaining scales, since $d\geq7$,
\begin{equation*}
\begin{split}
\sum_{i\geq0}
\mathbb{P}\left(
        z+\mathcal{T}_c\conn
        \mathcal{S}^{(i)}+A
\right)
&\lesssim
\mathrm{Cap}_{d-6}(A)
\sum_{i\geq0}r_i^{4-d}
\\
&=
\frac{
    \mathrm{Cap}_{d-6}(A)
}{
    \|z\|^{d-4}
}
\sum_{i\geq0}2^{-i(d-4)}
\lesssim
\frac{
    \mathrm{Cap}_{d-6}(A)
}{
    \|z\|^{d-4}
}.
\end{split}
\end{equation*}
Together with \eqref{eq-inner-annular-bound}, this gives
\begin{equation*}
    \mathbb{P}\left(
            z+\mathcal{S}+\mathcal{T}_c
        \conn A
    \right)
    \lesssim
    \frac{
        \mathrm{Cap}_{d-6}(A)
    }{
        \|z\|^{d-4}
    }.
\end{equation*}
This proves the upper bound of Proposition~\ref{prop-mixed-hitting-probability} for $m=1$.
\end{proof}

\subsection{Hitting by the sum of several random walks and one branching random walk}
\label{subsubsection-several-rw-one-brw}

We prove Proposition~\ref{prop-mixed-hitting-probability} for arbitrary
$m\geq1$, by using Proposition~\ref{prop-capacity-several-stopped-rw} for the lower bound, and by induction on the number of simple random walk ranges for the upper bound.

\begin{proof}

We first prove the lower bound. Suppose that $\|z\| \geq c_m \text{diam}(A)$. Fix a sufficiently small constant
$\varepsilon_m>0$ and set $
    \rho
    :=
    \left\lfloor
        \varepsilon_m\|z\|
    \right\rfloor$.
By choosing $\varepsilon_m$ sufficiently small, we may ensure that
\begin{equation*}
\begin{split}
\|z\|
&\geq
2\operatorname{diam}\left(
    A-\mathcal{S}^1(B_\rho)-\cdots-\mathcal{S}^m(B_\rho)
\right)
\end{split}
\end{equation*}
almost surely. Since for any $1\leq j\leq m$, $
    \mathcal{S}^j(B_\rho)
    \subseteq
    \mathcal{S}^j$,
we have
\begin{equation*}
\begin{split}
\mathbb{P}\left(
        z+\mathcal{S}^1+\cdots+\mathcal{S}^m+\mathcal{T}_c
    \conn A
\right)
&\geq
\mathbb{P}\left(
        z+\mathcal{T}_c
   \conn
        A-\mathcal{S}^1(B_\rho)-\cdots-
        \mathcal{S}^m(B_\rho)
\right).
\end{split}
\end{equation*}
Conditionally on the stopped random walk ranges,
Equations~\eqref{eq-introduction-brw-hitting} and \eqref{bcap-equiv-(d-4)cap} give
\begin{equation*}
\begin{split}
\mathbb{P}\left(
        z+\mathcal{T}_c
    \conn
        A-\mathcal{S}^1(B_\rho)-\cdots-
        \mathcal{S}^m(B_\rho)
\right)
&\gtrsim
\frac{
    \mathbb{E}\left[
        \mathrm{Cap}_{d-4}
        \left(
            A-\mathcal{S}^1(B_\rho)-\cdots-
            \mathcal{S}^m(B_\rho)
        \right)
    \right]
}{
    \|z\|^{d-2}
}.
\end{split}
\end{equation*}
By using symmetry and
Proposition~\ref{prop-capacity-several-stopped-rw}, we obtain
\begin{equation*}
\begin{split}
\mathbb{E}\left[
    \mathrm{Cap}_{d-4}
    \left(
        A-\mathcal{S}^1(B_\rho)-\cdots-
        \mathcal{S}^m(B_\rho)
    \right)
\right]
&\gtrsim
\rho^{2m}\,
\mathrm{Cap}_{d-2(m+2)}(A).
\end{split}
\end{equation*}
Since
$\rho\asymp\|z\|$, it follows that
\begin{equation*}
\begin{split}
\mathbb{P}\left(
        z+\mathcal{S}^1+\cdots+\mathcal{S}^m+\mathcal{T}_c
    \conn A
\right)
&\gtrsim
\frac{
    \|z\|^{2m}
    \mathrm{Cap}_{d-2(m+2)}(A)
}{
    \|z\|^{d-2}
}
=
\frac{
    \mathrm{Cap}_{d-2(m+2)}(A)
}{
    \|z\|^{d-2(m+1)}
}.
\end{split}
\end{equation*}

We prove the upper bound by induction on $m$. Suppose that $\|z\| \geq C_m \text{diam}(A)^{a_m}$, where $a_m$ 
is defined by $a_0=1$, $a_1=3$ and
\begin{equation}
\label{eq-a-m}
    a_m = m+1 + \sum_{i=1}^m \frac{1}{i} ,
\end{equation}
so $a_{m+1} = a_{m}+ (m+2)/(m+1)$.
(In particular its corresponds to $a_m=\alpha_{m,1}$ in Equation~\eqref{eq-alpha}.) The choice of \(a_m\) is explained by Equations~\eqref{condition-on-am-1} and \eqref{condition-on-am-2}.

Although the argument is
based on the same annular decomposition as in the case $m=1$, the
induction step does not follow directly from a decomposition into
annular excursions. An excursion at scale $r_i$ has diameter of order
$r_i$ and lies at distance of order $r_i$ from $z$, whereas the
induction hypothesis requires a nonlinear condition between
the starting point and the target set.

We therefore subdivide each annular excursion into local pieces whose
radius is of order $r_i^{1/a_m}$. This choice allows us to apply the
induction hypothesis to each local piece.

The case $m=1$ is the upper bound in the subsection above. Assume that the upper bound holds for some
$m\geq1$, and consider $m+1$ independent simple random walk ranges
$\mathcal S^1,\ldots,\mathcal S^{m+1}$, together with an independent
critical branching random walk range $\mathcal T_c$. Suppose that $d \geq 2m+7$ and $\|z\|\geq C_{m+1}
\operatorname{diam}(A)^{a_{m+1}}$.

By reflection symmetry,
\begin{equation*}
\begin{split}
\mathbb{P}\left(
        z+\mathcal{S}^1+\cdots+
        \mathcal{S}^{m+1}+\mathcal{T}_c
   \conn A
\right)
&=
\mathbb{P}\left(
        z+\mathcal{S}^2+\cdots+
        \mathcal{S}^{m+1}+\mathcal{T}_c
    \conn
        \mathcal{S}^1+A
\right).
\end{split}
\end{equation*}
We decompose $\mathcal{S}^1$ into annular pieces as in the case $m=1$. Let $r_i:=2^i\|z\|$,
and choose $I$ so that $
    r_{-I}
    \asymp
    \operatorname{diam}(A)^{a_m}$.
Denote the corresponding annular pieces of $\mathcal{S}^1$ by
$\mathcal{S}^{1,(i)}$. The union bound gives an inner contribution
and a sum of annular contributions:
\begin{equation*}
        \mathbb{P}\left(
        z+\mathcal{S}^2+\cdots+
        \mathcal{S}^{m+1}+\mathcal{T}_c
    \conn
        \mathcal{S}^1+A
\right)
\leq P_0+\sum_{i \geq -I}P_i
\end{equation*}
with 
\begin{equation*}
    \begin{split}
        P_0 & := \mathbb P\left(\mathcal{S}^1\conn\big(z+\mathcal{S}^2+\cdots+
            \mathcal{S}^{m+1}+\mathcal{T}_c\big)
        \cap B(z,2r_{-I})-A\right) \\
        P_i & :=\mathbb P\left(z+\mathcal{S}^2+\cdots+\mathcal{S}^{m+1}+\mathcal{T}_c\conn \mathcal{S}^{1,(i)}+A\right) .
    \end{split}
\end{equation*}

For the inner contribution, Equation~\eqref{eq-introduction-rw-hitting} and the
subadditivity of capacity imply
\begin{equation*}
\begin{split}
P_0
&\lesssim
\frac{
    \mathrm{Cap}_{d-2}(A)
}{
    \|z\|^{d-2}
}
\mathbb{E}\left[
    \left|
        \big(
            z+\mathcal{S}^2+\cdots+
            \mathcal{S}^{m+1}+\mathcal{T}_c
        \big)
        \cap B(z,2r_{-I})
    \right|
\right].
\end{split}
\end{equation*}
Let $G_{m+1}$ be the $(m+1)$-th convolution power of the Green's function $g$, then
\begin{equation*}
\begin{split}
\mathbb{E}\left[
    \left|
        \big(
            z+\mathcal{S}^2+\cdots+
            \mathcal{S}^{m+1}+\mathcal{T}_c
        \big)
        \cap B(z,2r_{-I})
    \right|
\right]
& \leq \sum_{x \in B(z,2r_{-I})} G_{m+1}(z,x).
\end{split}
\end{equation*}
Moreover, one can prove by induction that $G_{m+1}(z,x) \asymp \frac{1}{(1+\|z-x\|)^{d-2(m+1)}}$, therefore we have $\sum_{x \in B(z,2r_{-I})} G_{m+1}(z,x)\lesssim
r_{-I}^{2(m+1)}$.

Using \eqref{comparisonof2cap} and the last two equations, we get
\begin{equation*}
\begin{split}
P_0
&\lesssim
\frac{
    \operatorname{diam}(A)^{2(m+2)}
    r_{-I}^{2(m+1)}
    \mathrm{Cap}_{d-2(m+3)}(A)
}{
    \|z\|^{d-2}
}.
\end{split}
\end{equation*}
Since $
    r_{-I}\asymp \operatorname{diam}(A)^{a_m}$ and $ \|z\|\geq C_{m+1}\operatorname{diam}(A)^{a_{m+1}}$,
\begin{equation}\label{condition-on-am-1}
    \operatorname{diam}(A)^{2(m+2)}
    r_{-I}^{2(m+1)}
    \lesssim \operatorname{diam}(A)^{2(m+2+a_m(m+1))} \lesssim \|z\|^{\frac{2(m+2+a_m(m+1))}{a_{m+1}}}
    .
\end{equation}
And under the hypothesis that $a_{m+1}=a_m+(m+2)/(m+1)$,
\begin{equation}\label{condition-on-am-2}
    \frac{\|z\|^{\frac{2(m+2+a_m(m+1))}{a_{m+1}}}}{\|z\|^{d-2}}= \frac{1}{\|z\|^{d-2(m+2)}}. 
\end{equation}
Therefore,
\begin{equation}
\label{eq-inner-general-mixed}
    P_0
    \lesssim
    \frac{
        \mathrm{Cap}_{d-2(m+3)}(A)
    }{
        \|z\|^{d-2(m+2)}
    }.
\end{equation}

We next consider the annular terms. For every $i\geq-I$, set
\begin{equation*}
    \ell_i
    :=
    c_0r_i^{1/a_{m}},
\end{equation*}
where $c_0>0$ is sufficiently small and $a_m$ is the separation exponent appearing in the induction
hypothesis for $m$ simple random walks. We divide each annular excursion
of $\mathcal{S}^1$ into pieces stopped after travelling a distance
$\ell_i$, so that we can apply the induction hypothesis. The number of such pieces has expectation of order at most $
    1+\frac{r_i^2}{\ell_i^2}$.

For a random walk started at $x\in\Sigma_i$, define $
    \nu_0:=0$,
and for every $j\geq0$,
\begin{equation*}
    \nu_{j+1}
    :=
    \inf\left\{
        n>\nu_j:
        \big\|
            S_n^1-S_{\nu_j}^1
        \big\|>\ell_i
    \right\}.
\end{equation*}
Let
\begin{equation*}
    M_i
    :=
    \max\left\{
        j\geq0:
        \nu_j\leq\sigma_i^0
    \right\}.
\end{equation*}
Thus, $M_i+1$ is the number of pieces needed to cover
$\mathcal{S}_x^1[0,\sigma_i^0]$.

We claim that
\begin{equation}
\label{eq-number-small-annular-pieces-general}
    \sup_{x\in\Sigma_i}
    \mathbb E_x[M_i+1]
    \lesssim
    1+\frac{r_i^2}{\ell_i^2}.
\end{equation}
Indeed, for $n\geq0$, define
\begin{equation*}
    \mathcal N_n^{(\ell_i)}
    :=
    \max\left\{
        j\geq0:
        \nu_j\leq n
    \right\}.
\end{equation*}
On the event $\{qr_i^2\leq\sigma_i^0<(q+1)r_i^2\}$,
we have
\begin{equation*}
    M_i+1
    \leq
    \mathcal N_{(q+1)r_i^2}^{(\ell_i)}+1.
\end{equation*}
Consequently, by the Cauchy--Schwarz inequality,
\begin{equation*}
\begin{split}
\mathbb E_x[M_i+1]
&\leq
\sum_{q\geq0}
\mathbb E_x\Big[
    \big(
        \mathcal N_{(q+1)r_i^2}^{(\ell_i)}+1
    \big)
    \mathds{1}_{\{\sigma_i^0\geq qr_i^2\}}
\Big]
\\
&\leq
\sum_{q\geq0}
\mathbb E_x\Big[
    \big(
        \mathcal N_{(q+1)r_i^2}^{(\ell_i)}+1
    \big)^2
\Big]^{1/2}
\mathbb P_x\left(
    \sigma_i^0\geq qr_i^2
\right)^{1/2}.
\end{split}
\end{equation*}
By the second-moment estimate from Remark~\ref{remark-second-moment-stopped-pieces} (following the proof of Lemma~\ref{lemma-number-stopped-pieces}), we get
\begin{equation*}
\mathbb E_x\Big[
    \big(
        \mathcal N_n^{(\ell_i)}+1
    \big)^2
\Big]^{1/2}
\lesssim
1+\frac{n}{\ell_i^2}.
\end{equation*}
Moreover, the exit time from the annulus satisfies
\begin{equation*}
    \sup_{x\in\Sigma_i}
    \mathbb P_x\left(
        \sigma_i^0\geq qr_i^2
    \right)
    \lesssim
    e^{-cq}.
\end{equation*}
It follows that
\begin{equation*}
\begin{split}
\sup_{x\in\Sigma_i}
\mathbb E_x[M_i+1]
&\lesssim
\sum_{q\geq0}
\Big(
    1+
    \frac{(q+1)r_i^2}{\ell_i^2}
\Big)
e^{-cq/2} \lesssim
1+\frac{r_i^2}{\ell_i^2}.
\end{split}
\end{equation*}
This proves \eqref{eq-number-small-annular-pieces-general}.
For $P_i$, by the strong Markov property, a union bound, and the estimate above,
\begin{equation*}
\begin{split}
P_i
&\lesssim
p_i
\sup_{x\in\Sigma_i}
    \sum_{j\geq0}
    \mathbb P_x\left(
        \tau_i^j<\infty
    \right)
\Big(
    1+\frac{r_i^2}{\ell_i^2}
\Big)
\\
&\qquad\quad\times
\sup_{y\in\mathscr C_i}
\mathbb{P}\left(
        z+\mathcal{S}^2+\cdots+
        \mathcal{S}^{m+1}+\mathcal{T}_c
    \conn
        y+\widetilde{\mathcal S}(B_{2\ell_i})+A
\right),
\end{split}
\end{equation*}
where $\widetilde{\mathcal S}$ is an independent simple random walk
range and $
    p_i
    :=
    \mathbb P\left(
        \tau_i^0<\infty
    \right)$.
By transience of the walk,
\begin{equation*}
    \sup_{i\geq-I}
    \sup_{x\in\Sigma_i}
    \sum_{j\geq0}
\mathbb P_x(\tau_i^j<\infty)\lesssim1.
\end{equation*}
Then
\begin{equation}
\label{eq-general-annular-piece-reduction}
\begin{split}
P_i
&\lesssim
p_i
\Big(
    1+\frac{r_i^2}{\ell_i^2}
\Big)
\times
\sup_{y\in\mathscr C_i}
\mathbb{P}\left(
        z+\mathcal{S}^2+\cdots+
        \mathcal{S}^{m+1}+\mathcal{T}_c
    \conn
        y+\widetilde{\mathcal S}(B_{2\ell_i})+A
\right).
\end{split}
\end{equation}

For $y\in\mathscr C_i$, we have $
    \|z-y\|\asymp r_i$.
Moreover,
$
\operatorname{diam}\left(
    \widetilde{\mathcal S}(B_{2\ell_i})+A
\right)
\lesssim
\ell_i+\operatorname{diam}(A)$.
Choose $c_0$ sufficiently small and $r_{-I}$ large enough relative to
$\operatorname{diam}(A)^{a_m}$. Then $
    r_i
    \geq
    C_{m}
    \left(
        \ell_i+\operatorname{diam}(A)
    \right)^{a_m}$.
Therefore, the induction hypothesis gives
\begin{equation*}
\begin{split}
\sup_{y\in\mathscr C_i}
\mathbb{P}\left(
        z+\mathcal{S}^2+\cdots+
        \mathcal{S}^{m+1}+\mathcal{T}_c
    \conn
        y+\widetilde{\mathcal S}(B_{2\ell_i})+A
\right)
&\lesssim
\frac{
    \mathbb E\left[
        \mathrm{Cap}_{d-2(m+2)}
        \left(
            \widetilde{\mathcal S}(B_{2\ell_i})+A
        \right)
    \right]
}{
    r_i^{d-2(m+1)}
}.
\end{split}
\end{equation*}
Since $\ell_i\gtrsim\operatorname{diam}(A)$, the upper bound in
Equation~\eqref{eq-introduction-stopped-rw-capacity} gives
\begin{equation*}
\begin{split}
\mathbb E\left[
    \mathrm{Cap}_{d-2(m+2)}
    \left(
        \widetilde{\mathcal S}(B_{2\ell_i})+A
    \right)
\right]
&\lesssim
\ell_i^2\,
\mathrm{Cap}_{d-2(m+3)}(A).
\end{split}
\end{equation*}
Combining this estimate with
\eqref{eq-general-annular-piece-reduction}, we obtain
\begin{equation*}
\begin{split}
P_i
&\lesssim
p_i
\Big(
    1+\frac{r_i^2}{\ell_i^2}
\Big)
\frac{
    \ell_i^2\,
    \mathrm{Cap}_{d-2(m+3)}(A)
}{
    r_i^{d-2(m+1)}
}
\lesssim
p_i
r_i^{2m+4-d}\,
\mathrm{Cap}_{d-2(m+3)}(A).
\end{split}
\end{equation*}
Recall that $p_i \lesssim (r_i/\|z\|)^{d-2}$ when $-I\leq i <0$ and $p_i \lesssim 1$ when $i\geq 0$.
Hence,
\begin{equation*}
\begin{split}
\sum_{i=-I}^{-1}P_i
&\lesssim
\frac{
    \mathrm{Cap}_{d-2(m+3)}(A)
}{
    \|z\|^{d-2}
}
\sum_{i=-I}^{-1}r_i^{2m+2}
\lesssim
\frac{
    \mathrm{Cap}_{d-2(m+3)}(A)
}{
    \|z\|^{d-2m-4}
}.
\end{split}
\end{equation*}
For the remaining scales, since $d\geq2m+7$,
\begin{equation*}
\begin{split}
\sum_{i\geq0}P_i
&\lesssim
\mathrm{Cap}_{d-2(m+3)}(A)
\sum_{i\geq0}r_i^{2m+4-d}
\\
&=
\frac{
    \mathrm{Cap}_{d-2(m+3)}(A)
}{
    \|z\|^{d-2m-4}
}
\sum_{i\geq0}
2^{-i(d-2m-4)}
\lesssim
\frac{
    \mathrm{Cap}_{d-2(m+3)}(A)
}{
    \|z\|^{d-2m-4}
}.
\end{split}
\end{equation*}

Combining these estimates with
\eqref{eq-inner-general-mixed} proves the upper bound for $m+1$ and
completes the induction.
\end{proof}
\subsection{Capacity of a stopped branching random walk sausage}
\label{subsubsection-stopped-brw-sausage}

We now prove part~\textup{(ii)} of Theorem~\ref{theorem-main-capacity-introduction}, which gives the capacity estimate for the stopped critical branching random walk sausage. The upper bound follows from
Proposition~\ref{prop-mixed-hitting-probability}. For the lower bound, we compare the stopped critical branching random walk with the full
critical branching random walk.

\begin{lemma}
\label{lemma-stopped-to-full-brw}
There exists a constant $c_\star>1$ such that the following holds.
Let $K\subseteq\mathbb Z^d$ be finite and nonempty, and set $L:=\operatorname{dist}(0,K)$ and $\widetilde R:=\operatorname{dist}(\partial B_R,K)$.
Assume that
\begin{equation*}
    L\geq2\operatorname{diam}(K),
    \qquad
    \widetilde R\geq2\operatorname{diam}(K), \qquad \widetilde R\geq c_\star L.
\end{equation*}
Then we have
\begin{equation*}
    \mathbb P\left(
        \mathcal T_c(B_R)\conn K
    \right)
    \geq
    \frac12
    \mathbb P\left(
        \mathcal T_c\conn K
    \right).
\end{equation*}
\end{lemma}
\begin{figure}[tb]
    \centering
    \includegraphics[width=0.4\linewidth]{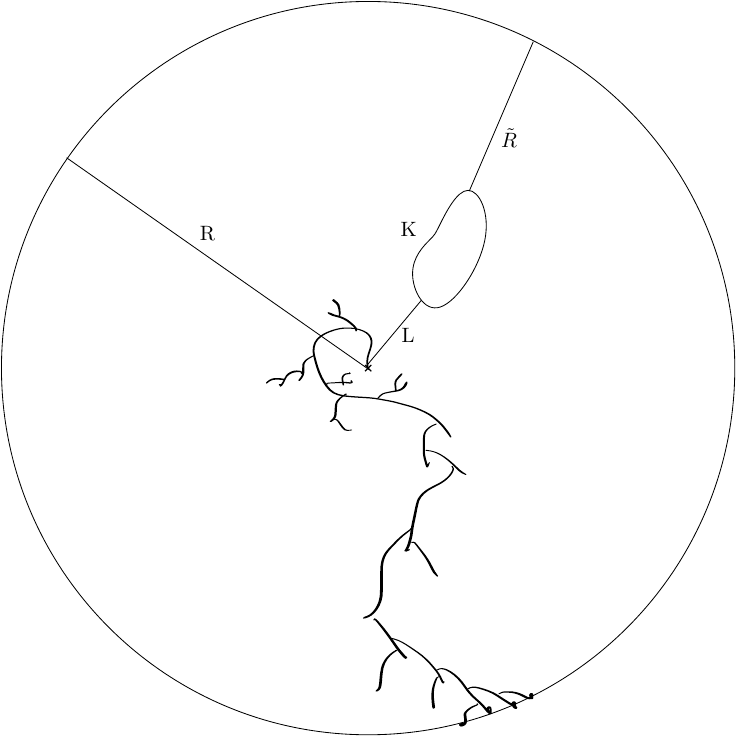}
    \caption{Illustration of the set $K$ in the ball.}
    \label{fig-stopped-to-full-brw}
\end{figure}
\begin{proof}
Let $\eta_R$ be the stopping line consisting of the particles that are
the first along their respective ancestral lines to reach
$\partial B_R$. More precisely,
\begin{equation*}
\begin{split}
\eta_R
:=
\bigl\{
    u:\,
    V_c(u)\in\partial B_R
    \text{ and }
    V_c(v)\notin\partial B_R
    \text{ for every }v<u
\bigr\}.
\end{split}
\end{equation*}

By the branching property and a union bound,
\begin{equation}
\label{eq-full-brw-decomposition}
\begin{split}
\mathbb P\left(
    \mathcal T_c\conn K
\right)
&\leq
\mathbb P\left(
    \mathcal T_c(B_R)\conn K
\right)+
\mathbb E\left[
    |\eta_R|
\right]
\sup_{x\in\partial B_R}
\mathbb P\left(
    \mathcal T_c^x\conn K \right).
\end{split}
\end{equation}

The many-to-one formula gives
\begin{equation*}
\begin{split}
\mathbb E\left[
    |\eta_R|
\right]
&=
\sum_{n\geq0}
\mathbb E\bigg[
    \sum_{|u|=n}
    \mathds{1}_{\{u\in\eta_R\}}
\bigg]=
\sum_{n\geq0}
\mathbf P_0\left(
    H_R=n
\right)=1,
\end{split}
\end{equation*}
where $H_R$ is the first time at which a simple random walk reaches
$\partial B_R$.

By Zhu's Equation~\eqref{eq-introduction-brw-hitting}, there exist constants
$c_1,c_2>0$ such that
\begin{equation*}
\mathbb P\left(
    \mathcal T_c\conn K
\right)
\geq
c_1
\frac{
    \operatorname{BCap}(K)
}{
    L^{d-2}
}
\end{equation*}
and
\begin{equation*}
\sup_{x\in\partial B_R}
\mathbb P\left(
    \mathcal T_c^x\conn K
\right)
\leq
c_2
\frac{
    \operatorname{BCap}(K)
}{
    \widetilde R^{d-2}
}.
\end{equation*}
Consequently,
\begin{equation*}
\begin{split}
\sup_{x\in\partial B_R}
\mathbb P\left(
    \mathcal T_c^x\conn K
\right)
&\leq
\frac{c_2}{c_1}
\left(
    \frac{L}{\widetilde R}
\right)^{d-2}
\mathbb P\left(
    \mathcal T_c\conn K
\right).
\end{split}
\end{equation*}
Choose $c_\star>1$ sufficiently large so that
$
    \frac{c_2}{c_1}
    c_\star^{-(d-2)}
    \leq
    \frac12$.
Equation~\eqref{eq-full-brw-decomposition} then gives
\begin{equation*}
\mathbb P\left(
    \mathcal T_c(B_R)\conn K
\right)
\geq
\frac12
\mathbb P\left(
    \mathcal T_c\conn K
\right).
\end{equation*}
\end{proof}

\begin{proof}
Proof of Equation~\eqref{eq-introduction-stopped-brw-capacity} in Theorem \ref{theorem-main-capacity-introduction}.

\textit{Upper bound.}
Suppose that $R \geq C_N \text{diam}(A)^{a_{N-2}}$ with $a_{N-2}$ defined in Equation~\eqref{eq-a-m}.
Let
$\mathcal S^1,\ldots,\mathcal S^{N-2}$ be independent simple random
walk ranges, independent of $\mathcal T_c$. Choose
$z_R\in\mathbb Z^d$ such that $
    \|z_R\|=c_NR$.
Since $
    \operatorname{diam}
    \left(
        \mathcal T_c(B_R)+A
    \right)
    \lesssim
    R+\operatorname{diam}(A)$,
we may choose $c_N$ and then $R$ sufficiently large so that the
starting point condition in
AOSS's Equation~\eqref{eq-introduction-sum-rw-hitting} is satisfied. Therefore,
\begin{equation*}
\begin{split}
\mathbb E\left[
    \operatorname{Cap}_{d-2(N-2)}
    \left(
        \mathcal T_c(B_R)+A
    \right)
\right]
&\lesssim
\|z_R\|^{d-2(N-2)}
\mathbb P\left(
        z_R+\mathcal S^1+\cdots+\mathcal S^{N-2}
    \conn
        \mathcal T_c(B_R)+A
\right).
\end{split}
\end{equation*}

By reflection symmetry and the inclusion
$\mathcal T_c(B_R)\subseteq\mathcal T_c$,
\begin{equation*}
\begin{split}
\mathbb P\left(
        z_R+\mathcal S^1+\cdots+\mathcal S^{N-2}
   \conn
        \mathcal T_c(B_R)+A
\right)
&\leq
\mathbb P\left(
        z_R+\mathcal S^1+\cdots+
        \mathcal S^{N-2}+\mathcal T_c
    \conn A
\right).
\end{split}
\end{equation*}

For $R$ sufficiently large, $
    \|z_R\|
    \geq
    C_{N}
    \operatorname{diam}(A)^{a_{N-2}}$.
Hence, Proposition~\ref{prop-mixed-hitting-probability}, applied with
$m=N-2$, gives
\begin{equation*}
\begin{split}
\mathbb P\left(
        z_R+\mathcal S^1+\cdots+
        \mathcal S^{N-2}+\mathcal T_c
   \conn A
\right)
&\lesssim
\frac{
    \operatorname{Cap}_{d-2N}(A)
}{
    \|z_R\|^{d-2(N-1)}
}.
\end{split}
\end{equation*}

Combining the preceding estimates, we obtain
\begin{equation*}
\begin{split}
\mathbb E\left[
    \operatorname{Cap}_{d-2(N-2)}
    \left(
        \mathcal T_c(B_R)+A
    \right)
\right]
&\lesssim
\|z_R\|^{d-2(N-2)}
\frac{
    \operatorname{Cap}_{d-2N}(A)
}{
    \|z_R\|^{d-2(N-1)}
}
\lesssim
R^2
\operatorname{Cap}_{d-2N}(A).
\end{split}
\end{equation*}

\medskip
\noindent
\textit{Lower bound.}
Suppose that $R\geq cN \text{diam}(A)$.
Let $b_N>1$ be sufficiently large and then choose $M_N>b_N$
sufficiently large, with both constants depending only on $N$ and
$c_\star$. Set
\begin{equation*}
    R_2:=\frac{R}{M_N}
    \qquad\text{and}\qquad
    R_1:=b_NR_2.
\end{equation*}
Since $R$ is sufficiently large, we may assume that $
\label{eq-stopped-brw-lower-scale-condition}
    R_2\geq2\operatorname{diam}(A)$.

Let
$\mathcal S^1,\ldots,\mathcal S^{N-2}$ be independent simple random
walk ranges, independent of $\mathcal T_c$, and choose
$z_R\in\mathbb Z^d$ with $
    \|z_R\|=c_NR$,
up to integer rounding, where $c_N$ is sufficiently large. AOSS's Equation~\eqref{eq-introduction-sum-rw-hitting}, applied with
$N-2$ simple random walks, gives
\begin{equation}
\label{eq-general-stopped-brw-lower-representation}
\begin{split}
\mathbb E\left[
    \operatorname{Cap}_{d-2(N-2)}
    \left(
        \mathcal T_c(B_R)+A
    \right)
\right]
&\gtrsim
\|z_R\|^{d-2(N-2)}
\mathbb P\left(
        z_R+\mathcal S^1+\cdots+\mathcal S^{N-2}
    \conn
        \mathcal T_c(B_R)+A
\right).
\end{split}
\end{equation}

Let
\begin{equation*}
    H_{R_1}
    :=
    \inf\left\{
        n\geq0:
        z_R+S_n^1\in B_{R_1}
    \right\}.
\end{equation*}
Conditionally on the entrance point
$x\in\partial B_{R_1}$, let
$\widetilde{\mathcal S}^1,\ldots,
\widetilde{\mathcal S}^{N-2}$
be fresh independent simple random walk ranges and set
\begin{equation*}
    K_x
    :=
    x+
    \sum_{j=1}^{N-2}
    \widetilde{\mathcal S}^j(B_{R_2})
    -A.
\end{equation*}

By the strong Markov property and by restricting all the auxiliary
walks to their stopped ranges,
\begin{equation}
\label{eq-general-stopped-brw-lower-markov}
\begin{split}
\mathbb P\left(
        z_R+\mathcal S^1+\cdots+\mathcal S^{N-2}
    \conn
        \mathcal T_c(B_R)+A
\right)
&\geq
\mathbb P\left(
    H_{R_1}<\infty
\right)
\inf_{x\in\partial B_{R_1}}
\mathbb P\left(
    \mathcal T_c(B_R)\conn K_x
\right).
\end{split}
\end{equation}

Since $\|z_R\|$ is a sufficiently large multiple of $R_1$,
Equation~\eqref{eq-introduction-rw-hitting} and the capacity estimate
for balls give
\begin{equation}
\label{eq-general-hit-intermediate-ball}
\mathbb P\left(
    H_{R_1}<\infty
\right)
\gtrsim
\frac{
    R_1^{d-2}
}{
    \|z_R\|^{d-2}
}.
\end{equation}

We next verify the assumptions of
Lemma~\ref{lemma-stopped-to-full-brw}. Since $0\in A$ and each stopped
random walk range contains the origin, we have $x\in K_x$. Moreover, $
    K_x
    \subseteq
    B\left(
        x,
        (N-2)R_2+\operatorname{diam}(A)
    \right)$,
and hence
\begin{equation*}
    \operatorname{diam}(K_x)
    \leq
    2(N-2)R_2+\operatorname{diam}(A).
\end{equation*}
Furthermore,
\begin{equation*}
\begin{split}
R_1-(N-2)R_2-\operatorname{diam}(A)
&\leq
\operatorname{dist}(0,K_x)\leq
R_1
\end{split}
\end{equation*}
and
\begin{equation*}
\begin{split}
\operatorname{dist}
\left(
    \partial B_R,K_x
\right)
&\geq
R-R_1-(N-2)R_2-\operatorname{diam}(A).
\end{split}
\end{equation*}

By first taking $b_N$ sufficiently large and then $M_N$ sufficiently
large, we may ensure that, uniformly in
$x\in\partial B_{R_1}$, $
    \operatorname{dist}(0,K_x)
    \geq
    2\operatorname{diam}(K_x)$
and $
    \operatorname{dist}
    \left(
        \partial B_R,K_x
    \right)
    \geq
    \max\left\{
        2\operatorname{diam}(K_x),
        c_\star\operatorname{dist}(0,K_x)
    \right\}$.

Lemma~\ref{lemma-stopped-to-full-brw}, applied conditionally on $K_x$,
therefore gives
\begin{equation*}
\begin{split}
\mathbb P\left(
    \mathcal T_c(B_R)\conn K_x
    \,\middle|\,
    K_x
\right)
&\gtrsim
\mathbb P\left(
    \mathcal T_c\conn K_x
    \,\middle|\,
    K_x
\right).
\end{split}
\end{equation*}
Since $
    \operatorname{dist}(0,K_x)
    \leq
    R_1$,
Equations \eqref{eq-introduction-brw-hitting} and \eqref{bcap-equiv-(d-4)cap} give
\begin{equation*}
\begin{split}
\mathbb P\left(
    \mathcal T_c\conn K_x\right)
&\gtrsim
\frac {\mathbb E\left[
        \operatorname{Cap}_{d-4}(K_x)
    \right]
}{
    R_1^{d-2}
}.
\end{split}
\end{equation*}
By translation and reflection invariance,
\begin{equation*}
\begin{split}
\mathbb E\left[
    \operatorname{Cap}_{d-4}(K_x)
\right]
&=
\mathbb E\bigg[
    \operatorname{Cap}_{d-4}
    \bigg(
        \sum_{j=1}^{N-2}
        \widetilde{\mathcal S}^j(B_{R_2})+A
    \bigg)
\bigg].
\end{split}
\end{equation*}
Consequently, Proposition~\ref{prop-capacity-several-stopped-rw},
applied with $N-2$ stopped simple random walk ranges, yields
\begin{equation}
\label{eq-general-stopped-rw-input-lower}
\mathbb E\left[
    \operatorname{Cap}_{d-4}(K_x)
\right]
\gtrsim
R_2^{2(N-2)}
\operatorname{Cap}_{d-2N}(A).
\end{equation}

Combining
\eqref{eq-general-stopped-brw-lower-representation},
\eqref{eq-general-stopped-brw-lower-markov},
\eqref{eq-general-hit-intermediate-ball}, and
\eqref{eq-general-stopped-rw-input-lower}, we obtain
\begin{equation*}
\begin{split}
&\mathbb E\left[
    \operatorname{Cap}_{d-2(N-2)}
    \left(
        \mathcal T_c(B_R)+A
    \right)
\right]
\\
&\qquad\gtrsim
\|z_R\|^{d-2(N-2)}
\frac{
    R_1^{d-2}
}{
    \|z_R\|^{d-2}
}
\frac{
    R_2^{2(N-2)}
    \operatorname{Cap}_{d-2N}(A)
}{
    R_1^{d-2}
}
=
\frac{
    R_2^{2(N-2)}
}{
    \|z_R\|^{2N-6}
}
\operatorname{Cap}_{d-2N}(A).
\end{split}
\end{equation*}
Since $R_2\asymp R$ and $\|z_R\|\asymp R$, it follows that
\begin{equation*}
\mathbb E\left[
    \operatorname{Cap}_{d-2(N-2)}
    \left(
        \mathcal T_c(B_R)+A
    \right)
\right]
\gtrsim
R^2
\operatorname{Cap}_{d-2N}(A).
\end{equation*}
Together with the upper bound, this concludes the proof of part~\textup{(ii)} of
Theorem~\ref{theorem-main-capacity-introduction}.
\end{proof}

\section{Hitting by two critical branching random walks}
\label{section8}

The goal of this section is to prove Proposition~\ref{prop-hitting-several-brw} for $p=2$.
The lower bound is easy and follows by combining Zhu's estimate \eqref{eq-introduction-brw-hitting} with part \textup{(ii)} in Theorem~\ref{theorem-main-capacity-introduction}.

\begin{proof}[Lower bound in Proposition~\ref{prop-hitting-several-brw} for $p=2$]
For the lower bound, let $\|z\|\geq C \text{diam}(A)$ and
$R=\lfloor c\|z\|\rfloor$ with $c$ sufficiently small. Conditioning on $\mathcal T_c^2(B_R)$ and applying Zhu's estimate \eqref{eq-introduction-brw-hitting} gives
\begin{equation*}
\begin{split}
\mathbb P\left(z+\mathcal T_c^1+\mathcal T_c^2\conn A
\right)
&\gtrsim
\frac{
    \mathbb E\left[
        \operatorname{Cap}_{d-4}
        \left(A-\mathcal T_c^2(B_R)\right)
    \right]
}{\|z\|^{d-2}}.
\end{split}
\end{equation*}
By reflection symmetry and the lower bound in \eqref{eq-introduction-stopped-brw-capacity} of Theorem \ref{theorem-main-capacity-introduction} for $N=4$, 
\begin{equation*}
\mathbb E\left[
    \operatorname{Cap}_{d-4}
    \left(A-\mathcal T_c^2(B_R)\right)
\right]
\gtrsim
R^2\operatorname{Cap}_{d-8}(A).
\end{equation*}
Since $R\asymp\|z\|$, this proves the required lower bound.
\end{proof}

The rest of this section focuses on the upper bound and contains an important geometric construction of the paper.
A branching random walk may enter and leave the same annulus through
many distinct ancestral lines, so the stopping-time decomposition used
for a single random walk has no direct analogue. We replace it with
successive entrance stopping lines. The resulting annular
waves have geometrically decreasing expected entrance populations.

\subsection{Annular decomposition into waves}
\label{section-waves}

Let us now present the decomposition into waves that is used in the proof of the upper bound of Proposition~\ref{prop-hitting-several-brw}. The main difference with the decomposition of a simple random walk is that, for a critical branching random walk, stopping lines replace stopping times.

For any $i\in\mathbb{Z}$, we define
\begin{equation*}
    r_i=2^i\|z\|.
\end{equation*}
Let $I$ be the maximal index such that $
    r_{-I}\geq 4\operatorname{diam}(A)^4\sqrt{\operatorname{diam}(A)}$.
For any $i\geq -I$, let
\begin{equation*}
    \mathscr{C}_i=B(z,r_{i+1})\setminus B(z,r_{i-1}),
    \qquad
    \Sigma_i=\partial B(z,r_i).
\end{equation*}

For any $i\geq -I$ and $k\geq0$, we define successively two stopping lines. First,
\begin{equation*}
    \begin{split}
        \eta_i^k=\bigl\{u\in\mathcal{T}_c^1:\;&V_c^1(u)\in\Sigma_i,\text{ and there exists }v\in\sigma_i^{k-1}\text{ such that }v<u,\\
        &\text{and }V_c^1(w)\notin\Sigma_i\text{ for every }w\in[v,u)\bigr\}.
    \end{split}
\end{equation*}
Then,
\begin{equation*}
    \begin{split}
        \sigma_i^k=\bigl\{u\in\mathcal{T}_c^1:\;&V_c^1(u)\in\partial\mathscr{C}_i,\text{ and there exists }v\in\eta_i^k\text{ such that }v<u,\\
        &\text{and }V_c^1(w)\notin\partial\mathscr{C}_i\text{ for every }w\in[v,u)\bigr\}.
    \end{split}
\end{equation*}
We use the convention $\sigma_i^{-1}=\{\emptyset\}$ for $i\neq0$, and we set $
    \eta_0^0=\{\emptyset\}$,
since the root belongs to $\Sigma_0$.

 We call the $k$-th wave at level $i$ the collection of branching random
walks started from the particles in $\eta_i^k$ and stopped upon first
hitting $\partial\mathscr{C}_i$. The line $\eta_i^k$ therefore represents
the particles that initiate the $k$-th wave, while $\sigma_i^k$ consists
of the first particles of this wave to hit $\partial\mathscr{C}_i$.
The descendants of the particles in $\sigma_i^k$ that return to
$\Sigma_i$ form the next entrance line $\eta_i^{k+1}$.

We denote by $\mathcal{T}_c^{1,\eta_i^k}(\mathscr{C}_i)$ the contribution of the branching random walks starting from the vertices of $\eta_i^k$ and stopped when they first hit $\partial\mathscr{C}_i$, and we write
\begin{equation*}
    \mathcal{T}_c^{(1,i)}
    =
    \underset{k\geq0}{\bigcup}
    \mathcal{T}_c^{1,\eta_i^k}(\mathscr{C}_i).
\end{equation*}
Then,
\begin{equation}\label{wave-covering}
    \mathcal{T}_c^1
    \subseteq
    \left(\mathcal{T}_c^1\cap B(z,r_{-I})\right)
    \cup
    \left(
    \underset{i\geq-I}{\bigcup}
    \mathcal{T}_c^{(1,i)}
    \right).
\end{equation}

For a stopping line $L$, we denote by $\mathcal{F}_L$ the sigma-field generated by the tree and by the spatial positions revealed before the line $L$. Conditional on $\mathcal{F}_L$, the branching random walks starting from the vertices of $L$ are independent, and the branching random walk starting from $u\in L$ has the law of a critical branching random walk started from $V_c^1(u)$.

\begin{figure}[tb]
         \centering
         \includegraphics[width=0.5\linewidth]{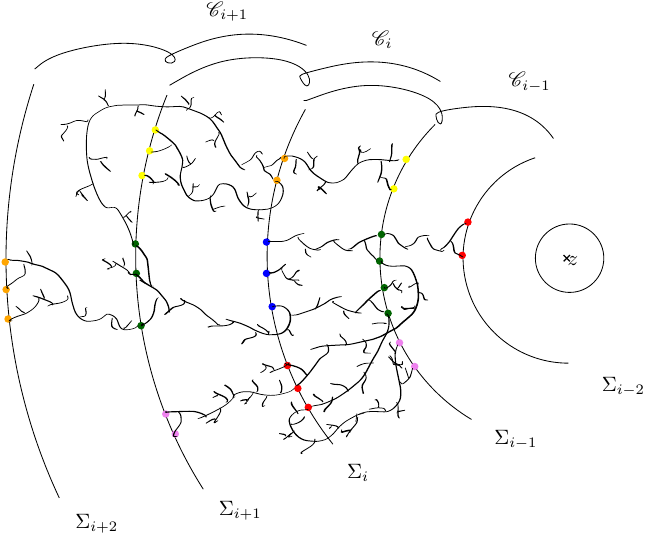}
         \caption{The first two waves at
    level $i$.}
         \label{waves}
     \end{figure}
Figure~\ref{waves} illustrates the first two waves of the
decomposition at the fixed level $i$. The blue particles on
$\Sigma_i$ form the initial entrance line $\eta_i^0$. Starting from
these particles, the green particles are the first descendants to
reach the boundary $
    \partial\mathscr C_i
    =
    \Sigma_{i-1}\cup\Sigma_{i+1}$;
they therefore form the exit line $\sigma_i^0$.

Some descendants of the particles in $\sigma_i^0$ subsequently return
to $\Sigma_i$. The first particles to do so form the next entrance
line $\eta_i^1$, represented in the figure by the red and orange
particles on $\Sigma_i$, according to whether their ancestral line
returns from $\Sigma_{i-1}$ or from $\Sigma_{i+1}$. Starting from
$\eta_i^1$, the first descendants to reach
$\partial\mathscr C_i$ form the next exit line $\sigma_i^1$,
represented by the yellow and purple particles. Iterating this
construction produces the successive entrance and exit lines
$\eta_i^k$ and $\sigma_i^k$, for $k\geq0$.

\subsection{Geometric decay of successive waves}

The main property of this decomposition is that the expected size of the successive entrance lines decreases geometrically.

\begin{lemma}
\label{lemma-geometric-decay-waves}
There exists a constant $\gamma_d>0$, independent of $i$, such that,
for every $i\geq-I$ and every $l\geq1$,
\begin{equation}
\label{eq-geometric-decay-waves}
\mathbb{E}\left[
    \#\eta_i^l
    \,\middle|\,
    \mathcal{F}_{\eta_i^{l-1}}
\right]
\leq
\left(
    1-\gamma_d
\right)
\#\eta_i^{l-1}.
\end{equation}
\end{lemma}
\begin{proof}
For every $u\in\sigma_i^{l-1}$, let $\xi^{(u)}$ denote the number of
descendants of $u$ that first reach $\Sigma_i$. More precisely,
\begin{equation*}
\begin{split}
\xi^{(u)}
:=
\sum_{v\succeq u}
\mathds{1}_{\left\{
    V_c^1(v)\in\Sigma_i,\,
    V_c^1(w)\notin\Sigma_i
    \text{ for every }u\preceq w\prec v
\right\}}.
\end{split}
\end{equation*}
By the definition of the successive waves, $
    \#\eta_i^l
    =
    \sum_{u\in\sigma_i^{l-1}}
    \xi^{(u)}$.

We divide the stopping line $\sigma_i^{l-1}$ according to the
boundary through which its particles exit the annulus. Set
\begin{equation*}
    L_{i+1}^{l-1}
    :=
    \left\{
        u\in\sigma_i^{l-1}:
        V_c^1(u)\in\Sigma_{i+1}
    \right\}, \qquad  L_{i-1}^{l-1}
    :=
    \left\{
        u\in\sigma_i^{l-1}:
        V_c^1(u)\in\Sigma_{i-1}
    \right\}.
\end{equation*}
Then $
    \sigma_i^{l-1}
    =
    L_{i+1}^{l-1}
    \bigcup
    L_{i-1}^{l-1}$.
Let $H_{r_j}$ denote the hitting time of $\Sigma_j$ by a simple
random walk. By the branching property at the stopping line
$\sigma_i^{l-1}$, conditionally on
$\mathcal{F}_{\sigma_i^{l-1}}$, the descendant trees rooted at the
vertices of $\sigma_i^{l-1}$ are independent critical branching
random walks. Therefore,
\begin{equation}
\label{eq-eta-conditional-sigma}
\begin{split}
\mathbb{E}\left[
    \#\eta_i^l
    \,\middle|\,
    \mathcal{F}_{\sigma_i^{l-1}}
\right]
&=
\sum_{u\in L_{i+1}^{l-1}}
\mathbb{E}\left[
    \xi^{(u)}
    \,\middle|\,
    \mathcal{F}_{\sigma_i^{l-1}}
\right]
+
\sum_{u\in L_{i-1}^{l-1}}
\mathbb{E}\left[
    \xi^{(u)}
    \,\middle|\,
    \mathcal{F}_{\sigma_i^{l-1}}
\right].
\end{split}
\end{equation}
Let $u\in L_{i+1}^{l-1}$. Conditionally on
$\mathcal{F}_{\sigma_i^{l-1}}$, the branching random walk rooted at
$u$ starts from $V_c^1(u)\in\Sigma_{i+1}$. Let $Z_n^{(u)}$ denote the
number of descendants of $u$ at genealogical distance $n$. Since the
offspring distribution is critical, $\mathbb{E}\left[Z_n^{(u)}\right]=1$.
The many-to-one formula gives
\begin{equation*}
\begin{split}
\mathbb{E}\left[
    \xi^{(u)}
    \,\middle|\,
    \mathcal{F}_{\sigma_i^{l-1}}
\right]
&=
\sum_{n\geq0}
\mathbb{E}\Bigg[
    \sum_{\substack{
        v\succeq u\\
        |v|-|u|=n
    }}
    \mathds{1}_{\left\{
        V_c^1(v)\in\Sigma_i,\,
        V_c^1(w)\notin\Sigma_i
        \text{ for every }u\preceq w\prec v
    \right\}}
    \;\Bigg|\;
    \mathcal{F}_{\sigma_i^{l-1}}
\Bigg]
\\
&=
\sum_{n\geq0}
\mathbf{P}_{V_c^1(u)}
\left(
    H_{r_i}=n
\right)
=
\mathbf{P}_{V_c^1(u)}
\left(
    H_{r_i}<\infty
\right).
\end{split}
\end{equation*}
We claim that there exists a constant $c_d^*\in(0,1)$, independent of
$i$, such that
\begin{equation}
\label{eq-uniform-outer-return}
    \sup_{x\in\Sigma_{i+1}}
    \mathbf{P}_x\left(
        H_{r_i}<\infty
    \right)
    \leq
    c_d^*.
\end{equation}
Indeed, by a Gambler's ruin estimate, for $x\notin B(z,r_i)$
\begin{equation*}
\begin{split}
\mathbf{P}_x\left(
    H_{r_i}<\infty
\right)
&\leq
\frac{
    g(x-z)
}{
    \displaystyle
    \inf_{y\in\Sigma_i}g(y-z)
},
\end{split}
\end{equation*}
which by Green function asymptotic $
    g(y)
    =
    a_d\|y\|^{2-d}
    +
    \mathcal{O}\left(
        \|y\|^{-d}
    \right)$,
and the relation $r_{i+1}=2r_i$ gives
\begin{equation*}
\begin{split}
\sup_{x\in\Sigma_{i+1}}
\mathbf{P}_x\left(
    H_{r_i}<\infty
\right)
&\leq
2^{2-d}
+
\mathcal{O}\left(
    r_i^{-1}
\right).
\end{split}
\end{equation*}
Consequently, there exists $c_d^*<1$, independent of $i$, such that
\eqref{eq-uniform-outer-return} holds.
On the other hand, if $u\in L_{i-1}^{l-1}$, the many-to-one formula
similarly gives
\begin{equation}
\label{eq-xi-inner-boundary}
\begin{split}
\mathbb{E}\left[
    \xi^{(u)}
    \,\middle|\,
    \mathcal{F}_{\sigma_i^{l-1}}
\right]
&=
\mathbf{P}_{V_c^1(u)}
\left(
    H_{r_i}<\infty
\right)
\leq
1.
\end{split}
\end{equation}

It follows from
\eqref{eq-eta-conditional-sigma},
\eqref{eq-uniform-outer-return}, and
\eqref{eq-xi-inner-boundary} that
\begin{equation}
\label{eq-return-line-estimate-detailed}
\begin{split}
\mathbb{E}\left[
    \#\eta_i^l
    \,\middle|\,
    \mathcal{F}_{\sigma_i^{l-1}}
\right]
&\leq
c_d^*\#L_{i+1}^{l-1}
+
\#L_{i-1}^{l-1}.
\end{split}
\end{equation}
We now estimate the conditional expected number of particles in
$L_{i+1}^{l-1}$.

Conditionally on $\mathcal{F}_{\eta_i^{l-1}}$, every particle
$u\in\eta_i^{l-1}$ starts an independent critical branching random
walk from $V_c^1(u)\in\Sigma_i$, stopped when it first reaches either
$\Sigma_{i-1}$ or $\Sigma_{i+1}$. By the many-to-one formula,
\begin{equation}
\label{eq-expected-outer-line}
\begin{split}
  \mathbb{E}\left[
    \#L_{i+1}^{l-1}
    \,\middle|\,
    \mathcal{F}_{\eta_i^{l-1}}
\right]
&=
\sum_{u\in\eta_i^{l-1}}
\mathbf{P}_{V_c^1(u)}
\left(
    H_{r_{i+1}}<H_{r_{i-1}}
\right).
\end{split}
\end{equation}
For $x\in\Sigma_i$, the standard exit estimate from an annulus gives
\begin{equation*}
\begin{split}
\mathbf{P}_x\left(
    H_{r_{i+1}}<H_{r_{i-1}}
\right)
&=
\frac{
    r_{i-1}^{2-d}
    -
    \|x-z\|^{2-d}
    +
    \mathcal{O}\left(
        r_{i-1}^{1-d}
    \right)
}{
    r_{i-1}^{2-d}
    -
    r_{i+1}^{2-d}
}=\frac{
    1-2^{2-d}
    +
    \mathcal{O}\left(
        r_{i-1}^{-1}
    \right)
}{
    1-4^{2-d}
}.
\end{split}
\end{equation*}

The limiting ratio is strictly positive. Hence, by a different constant, if necessary, there exists $p_d>0$, independent of $i$, such
that
\begin{equation}
\label{eq-uniform-outer-exit-probability}
    \inf_{x\in\Sigma_i}
    \mathbf{P}_x\left(
        H_{r_{i+1}}<H_{r_{i-1}}
    \right)
    \geq
    p_d.
\end{equation}
Combining
\eqref{eq-expected-outer-line} and
\eqref{eq-uniform-outer-exit-probability}, we obtain
\begin{equation}
\label{eq-lower-outer-line}
\mathbb{E}\left[
    \#L_{i+1}^{l-1}
    \,\middle|\,
    \mathcal{F}_{\eta_i^{l-1}}
\right]
\geq
p_d\#\eta_i^{l-1}.
\end{equation}
Moreover, since the random walk exits the finite annulus almost
surely, the many-to-one formula gives
\begin{equation*}
\begin{split}
\mathbb{E}\left[
    \#L_{i+1}^{l-1}
    +
    \#L_{i-1}^{l-1}
    \,\middle|\,
    \mathcal{F}_{\eta_i^{l-1}}
\right]
&=
\sum_{u\in\eta_i^{l-1}}
\mathbf{P}_{V_c^1(u)}
\left(
    H_{r_{i-1}}\wedge H_{r_{i+1}}
    <\infty
\right)=
\#\eta_i^{l-1}.
\end{split}
\end{equation*}
Consequently,
\begin{equation}
\label{eq-upper-inner-line}
\begin{split}
\mathbb{E}\left[
    \#L_{i-1}^{l-1}
    \,\middle|\,
    \mathcal{F}_{\eta_i^{l-1}}
\right]
&=
\#\eta_i^{l-1}
-
\mathbb{E}\left[
    \#L_{i+1}^{l-1}
    \,\middle|\,
    \mathcal{F}_{\eta_i^{l-1}}
\right]
\leq
\left(
    1-p_d
\right)
\#\eta_i^{l-1}.
\end{split}
\end{equation}
Taking the conditional expectation of
\eqref{eq-return-line-estimate-detailed} with respect to
$\mathcal{F}_{\eta_i^{l-1}}$, and using
\eqref{eq-lower-outer-line} and
\eqref{eq-upper-inner-line}, we get
\begin{equation*}
\begin{split}
\mathbb{E}\left[
    \#\eta_i^l
    \,\middle|\,
    \mathcal{F}_{\eta_i^{l-1}}
\right]
&\leq
c_d^*
\mathbb{E}\left[
    \#L_{i+1}^{l-1}
    \,\middle|\,
    \mathcal{F}_{\eta_i^{l-1}}
\right]
+
\mathbb{E}\left[
    \#L_{i-1}^{l-1}
    \,\middle|\,
    \mathcal{F}_{\eta_i^{l-1}}
\right]
\\
&=
\#\eta_i^{l-1}
-
\left(
    1-c_d^*
\right)
\mathbb{E}\left[
    \#L_{i+1}^{l-1}
    \,\middle|\,
    \mathcal{F}_{\eta_i^{l-1}}
\right]
\\
&\leq
\left[
    1-
    p_d\left(
        1-c_d^*
    \right)
\right]
\#\eta_i^{l-1}.
\end{split}
\end{equation*}
Set $
    \gamma_d
    :=
    p_d\left(
        1-c_d^*
    \right)>0$.
We conclude that
\begin{equation*}
\mathbb{E}\left[
    \#\eta_i^l
    \,\middle|\,
    \mathcal{F}_{\eta_i^{l-1}}
\right]
\leq
\left(
    1-\gamma_d
\right)
\#\eta_i^{l-1},
\end{equation*}
which proves the lemma.
\end{proof}

By taking the expectation in Equation~\eqref{eq-geometric-decay-waves}, we get that,
\begin{equation*}
    \sum_{k \geq 0}\mathbb{E}\left[\#\eta_i^k\right]
    \leq \sum_{k \geq 0}
    (1-\gamma_d)
    \mathbb{E}\left[\#\eta_i^{k-1}\right]\leq \sum_{k \geq 0}
    (1-\gamma_d)^k
    \mathbb{E}\left[\#\eta_i^0\right]\leq \frac{1}{\gamma_d}\left[\#\eta_i^0\right].
\end{equation*}

Finally, the many-to-one formula gives
\begin{equation*}
    \mathbb{E}\left[\#\eta_i^0\right]
    =
    \mathbf P_0(H_{r_i}<\infty) \lesssim
    \frac{r_i^{d-2}}{\|z\|^{d-2}}
    \mathds{1}_{\{-I\leq i<0\}}
    +
    \mathds{1}_{\{i\geq0\}}.
\end{equation*}
Combining the last two equations, we obtain
\begin{equation}\label{total-number-waves}
    \sum_{k\geq0}
    \mathbb{E}\left[\#\eta_i^k\right]
    \lesssim
    \frac{r_i^{d-2}}{\|z\|^{d-2}}
    \mathds{1}_{\{-I\leq i<0\}}
    +
    \mathds{1}_{\{i\geq0\}}.
\end{equation}

\subsection{Capacity of the waves and upper bound}

\begin{proof}[Upper bound in Proposition~\ref{prop-hitting-several-brw} for $p=2$]
We first estimate the capacity of one annular level.
Using the subadditivity of the $(d-4)$-capacity and the branching property at the stopping lines $\eta_i^k$, we get
\begin{equation*}
    \begin{split}
        \mathbb{E}\left[
            \mathrm{Cap}_{d-4}
            \left(\mathcal{T}_c^{(1,i)}+A\right)
        \right]
        &\leq
        \sum_{k\geq0}
        \mathbb{E}\left[
            \mathrm{Cap}_{d-4}
            \left(
                \mathcal{T}_c^{1,\eta_i^k}(\mathscr{C}_i)+A
            \right)
        \right]\\
        &\leq
        \sum_{k\geq0}
        \mathbb{E}\left[\#\eta_i^k\right]
        \sup_{x\in\Sigma_i}
        \mathbb{E}\left[
            \mathrm{Cap}_{d-4}
            \left(
                \mathcal{T}_c^{1,x}(\mathscr{C}_i)+A
            \right)
        \right].
    \end{split}
\end{equation*}
For any $x\in\Sigma_i$, $
    \mathcal{T}_c^{1,x}(\mathscr{C}_i)
    \subseteq
    \mathcal{T}_c^{1,x}(B(x,3r_i))$.
Using the translation invariance of the capacity and, since $r_i \gtrsim \text{diam}(A)^4\sqrt{\text{diam}(A)}=\text{diam}(A)^{a_2}$, the upper bound in Equation~\eqref{eq-introduction-stopped-brw-capacity} for $N=4$, we obtain
\begin{equation*}
    \sup_{x\in\Sigma_i}
    \mathbb{E}\left[
        \mathrm{Cap}_{d-4}
        \left(
            \mathcal{T}_c^{1,x}(\mathscr{C}_i)+A
        \right)
    \right]
    \lesssim
    r_i^2\mathrm{Cap}_{d-8}(A).
\end{equation*}
Together with Equation~\eqref{total-number-waves}, this gives
\begin{equation}\label{capacity-one-wave-level}
    \mathbb{E}\left[
        \mathrm{Cap}_{d-4}
        \left(\mathcal{T}_c^{(1,i)}+A\right)
    \right]
    \lesssim
    r_i^2\mathrm{Cap}_{d-8}(A)
    \left(
        \frac{r_i^{d-2}}{\|z\|^{d-2}}
        \mathds{1}_{\{-I\leq i<0\}}
        +
        \mathds{1}_{\{i\geq0\}}
    \right).
\end{equation}
We are now ready to prove the upper bound of Proposition~\ref{prop-hitting-several-brw} for $p=2$. Assume that $\|z\|\geq C_2\,\text{diam}(A)^{\beta_2}$, where by definition $\beta_2=13/2$. 

By the symmetry of the jump distribution,
\begin{equation*}
\begin{split}
\mathbb{P}\left(
        z+\mathcal{T}_c^1+\mathcal{T}_c^2
    \conn A
\right)
&=
\mathbb{P}\left(
        z+\mathcal{T}_c^2
   \conn
        \mathcal{T}_c^1+A
\right).
\end{split}
\end{equation*}
By the covering in Equation~\eqref{wave-covering} and a union bound,
\begin{equation}
\label{eq-two-brw-union-bound}
\begin{split}
\mathbb{P}\left(
        z+\mathcal{T}_c^1+\mathcal{T}_c^2
    \conn A
\right)
&\leq
P_{0}
+
\sum_{i\geq-I}P_i,
\end{split}
\end{equation}
where
\begin{equation*}
\begin{split}
P_{0}
:=
\mathbb{P}\left(
        z+\mathcal{T}_c^2\conn
            \mathcal{T}_c^1\cap B(z,r_{-I})+A
\right), \qquad P_i
:=
\mathbb{P}\left(
        z+\mathcal{T}_c^2 \conn
        \mathcal{T}_c^{(1,i)}+A
\right).
\end{split}
\end{equation*}

For the first term, we can see that
\begin{equation}
\label{eq-inner-two-brw-inclusion}
\begin{split}
P_{0}
&\leq
\mathbb{P}\left(
    \mathcal{T}_c^1
   \conn
        \left(
            z+\mathcal{T}_c^2
        \right)
        \cap B(z,2r_{-I})-A
\right).
\end{split}
\end{equation}
Since the
conditions of Equation~\eqref{eq-introduction-brw-hitting} are satisfied,
\begin{equation}
\label{eq-inner-two-brw-capacity}
    P_{0}
    \lesssim
    \frac{
        \mathbb{E}\left[
            \mathrm{Cap}_{d-4}\left(\left(
            z+\mathcal{T}_c^2
        \right)
        \cap B(z,2r_{-I})-A\right)
        \right]
    }{
        \|z\|^{d-2}
    }.
\end{equation}

By Equation~\eqref{capofsum},
\begin{equation*}
\begin{split}
\mathrm{Cap}_{d-4}\left(\left(
            z+\mathcal{T}_c^2
        \right)
        \cap B(z,2r_{-I})-A\right)
&\leq
\mathrm{Cap}_{d-4}(-A)
\left|
    \left(
        z+\mathcal{T}_c^2
    \right)
    \cap B(z,2r_{-I})
\right|.
\end{split}
\end{equation*}
Furthermore,
\begin{equation*}
\begin{split}
\mathbb{E}\left[
    \left|
        \left(
            z+\mathcal{T}_c^2
        \right)
        \cap B(z,2r_{-I})
    \right|
\right]
&=
\mathbb{E}\left[
    \left|
        \mathcal{T}_c^2
        \cap B(0,2r_{-I})
    \right|
\right]
\leq
\sum_{x\in B(0,2r_{-I})}g(x)
\lesssim
r_{-I}^2.
\end{split}
\end{equation*}
Using now equations \eqref{eq-capacity-invariances} and \eqref{comparisonof2cap},
\begin{equation*}
\begin{split}
\mathrm{Cap}_{d-4}(-A)
&=
\mathrm{Cap}_{d-4}(A)\lesssim
\operatorname{diam}(A)^4\mathrm{Cap}_{d-8}(A),
\end{split}
\end{equation*}
we obtain
\begin{equation*}
\begin{split}
\mathbb{E}\left[
    \mathrm{Cap}_{d-4}\left(\left(
            z+\mathcal{T}_c^2
        \right)
        \cap B(z,2r_{-I})-A\right)
\right]
&\lesssim
\operatorname{diam}(A)^4r_{-I}^2
\mathrm{Cap}_{d-8}(A)\lesssim
\operatorname{diam}(A)^{13}\mathrm{Cap}_{d-8}(A).
\end{split}
\end{equation*}
Because $
    \operatorname{diam}(A)^{13}\lesssim\|z\|^2$,
Equation~\eqref{eq-inner-two-brw-capacity} yields
\begin{equation}
\label{eq-inner-two-brw-final}
    P_{0}
    \lesssim
    \frac{
        \mathrm{Cap}_{d-8}(A)
    }{
        \|z\|^{d-4}
    }.
\end{equation}

\noindent
\textit{The annular contributions. }
Observe that, for any $i \geq -I$, we can always split $\mathcal{T}_c^{(1,i)}$ into a finite number of pieces, independent of $z$ and $A$, with diameter at most $c_0 \,r_i$, where $c_0>0$ is a sufficiently small constant depending only on the
dimension. We may therefore apply
Equation~\eqref{eq-introduction-brw-hitting} to each piece. Using the monotonicity of capacity, we obtain
\begin{equation}
\label{eq-annular-hit-from-capacity}
\begin{split}
P_i
&\lesssim
\frac{
    \mathbb{E}\left[
        \mathrm{Cap}_{d-4}
        \left(
            \mathcal{T}_c^{(1,i)}+A
        \right)
    \right]
}{
    r_i^{d-2}
}.
\end{split}
\end{equation}

By Equation~\eqref{capacity-one-wave-level},
\begin{equation*}
\begin{split}
\mathbb{E}\left[
    \mathrm{Cap}_{d-4}
    \left(
        \mathcal{T}_c^{(1,i)}+A
    \right)
\right]
&\lesssim
r_i^2\mathrm{Cap}_{d-8}(A)
\left[
    \frac{
        r_i^{d-2}
    }{
        \|z\|^{d-2}
    }
    \mathds{1}_{\{-I\leq i<0\}}
    +
    \mathds{1}_{\{i\geq0\}}
\right].
\end{split}
\end{equation*}
Substituting this estimate into
\eqref{eq-annular-hit-from-capacity}, we find that
\begin{equation}
\label{eq-annular-two-brw-bound}
P_i
\lesssim
\mathrm{Cap}_{d-8}(A)
\left[
    \frac{
        r_i^2
    }{
        \|z\|^{d-2}
    }
    \mathds{1}_{\{-I\leq i<0\}}
    +
    r_i^{4-d}
    \mathds{1}_{\{i\geq0\}}
\right].
\end{equation}

For the negative scales,
\begin{equation}\label{eq-annular-two-brw-final1}
\begin{split}
\sum_{i=-I}^{-1}P_i
&\lesssim
\frac{
    \mathrm{Cap}_{d-8}(A)
}{
    \|z\|^{d-2}
}
\sum_{i=-I}^{-1}r_i^2
\lesssim
\frac{
    \mathrm{Cap}_{d-8}(A)
}{
    \|z\|^{d-4}
}.
\end{split}
\end{equation}
For the nonnegative scales, since $d\geq9$,
\begin{equation}\label{eq-annular-two-brw-final2}
\begin{split}
\sum_{i\geq0}P_i
&\lesssim
\mathrm{Cap}_{d-8}(A)
\sum_{i\geq0}r_i^{4-d}
=
\frac{
    \mathrm{Cap}_{d-8}(A)
}{
    \|z\|^{d-4}
}
\sum_{i\geq0}2^{-i(d-4)}
\lesssim
\frac{
    \mathrm{Cap}_{d-8}(A)
}{
    \|z\|^{d-4}
}.
\end{split}
\end{equation}

Finally, Equations~\eqref{eq-two-brw-union-bound},
\eqref{eq-inner-two-brw-final}, \eqref{eq-annular-two-brw-final1} and
\eqref{eq-annular-two-brw-final2} imply that
\begin{equation*}
\begin{split}
\mathbb{P}\left(
        z+\mathcal{T}_c^1+\mathcal{T}_c^2
   \conn A
\right)
&\lesssim
\frac{
    \mathrm{Cap}_{d-8}(A)
}{
    \|z\|^{d-4}
}.
\end{split}
\end{equation*}
This proves the upper bound in Proposition~\ref{prop-hitting-several-brw}.
\end{proof}

\section{Several critical branching random walks}
\label{section-several-brw}

We now extend the preceding result to the sum of an arbitrary number
of independent critical branching random walk ranges, and prove Proposition~\ref{prop-hitting-several-brw}.

The lower bound uses only the lower capacity estimate for one stopped
branching random walk sausage. 
For the upper bound, the proof
proceeds by induction on the number of branching random walks. 
We reuse the
annular-wave decomposition introduced in
Section~\ref{section8}. 
However, as in the proof of
Proposition~\ref{prop-mixed-hitting-probability}, an additional local
refinement is required in order to preserve the starting point
condition in the induction hypothesis.

\subsection{Lower bound}
\begin{proof}[Proof of the lower bound]
We proceed by induction on $p$. The case $p=2$ was proved in
Section~\ref{section8}. Assume that the lower bound holds for the sum
of $p-1$ critical branching random walks under the linear starting
point condition.

Choose $
    R
    :=
    \left\lfloor
        \varepsilon_p\|z\|
    \right\rfloor$,
where $\varepsilon_p>0$ is sufficiently small. Restricting the last
branching random walk to its range before leaving $B_R$, we obtain
\begin{equation*}
\begin{split}
\mathbb P\left(
        z+\mathcal T_c^1+\cdots+\mathcal T_c^p
   \conn A
\right)
&\geq
\mathbb P\left(
            z+\mathcal T_c^1+\cdots+\mathcal T_c^{p-1}\conn
            A-\mathcal T_c^p(B_R)
    \right).
\end{split}
\end{equation*}

By choosing $\varepsilon_p$ sufficiently small, the point $z$
satisfies the linear starting point condition required for the
lower-bound induction hypothesis. Hence,
\begin{equation}
\label{eq-lower-several-brw-induction}
\begin{split}
\mathbb P\left(
        z+\mathcal T_c^1+\cdots+\mathcal T_c^p
   \conn A
\right)
&\gtrsim
\frac{
    \mathbb E\left[
        \operatorname{Cap}_{d-4(p-1)}
        \left(
            A-\mathcal T_c^p(B_R)
        \right)
    \right]
}{
    \|z\|^{d-2(p-1)}
}.
\end{split}
\end{equation}

By the lower bound in
Equation~\eqref{eq-introduction-stopped-brw-capacity} of theorem \ref{theorem-main-capacity-introduction},
\begin{equation}
\label{eq-lower-several-brw-stopped-capacity}
\begin{split}
\mathbb E\left[
    \operatorname{Cap}_{d-4(p-1)}
    \left(
        A-\mathcal T_c^p(B_R)
    \right)
\right]
&\gtrsim
R^2
\operatorname{Cap}_{d-4p}(A).
\end{split}
\end{equation}

Combining
\eqref{eq-lower-several-brw-induction} and
\eqref{eq-lower-several-brw-stopped-capacity}, and using
$R\asymp\|z\|$, we obtain
\begin{equation*}
\begin{split}
\mathbb P\left(
        z+\mathcal T_c^1+\cdots+\mathcal T_c^p
  \conn A
\right)
&\gtrsim
\frac{
    R^2\operatorname{Cap}_{d-4p}(A)
}{
    \|z\|^{d-2(p-1)}
}
\asymp
\frac{
    \operatorname{Cap}_{d-4p}(A)
}{
    \|z\|^{d-2p}
}.
\end{split}
\end{equation*}
This proves the lower bound.
\end{proof}

\subsection{A local refinement of the annular decomposition}
\label{subsection-local-block-decomposition}

We now prove the upper bound by induction on $p$. As explained above,
the annular-wave decomposition alone is not sufficient to apply the
induction hypothesis. Each annular wave must first be subdivided into
smaller local pieces in order to verify the required starting point
condition.

In the proof for two branching random walks, an entire annular wave
can be used directly because the one-branching-random-walk hitting
estimate \eqref{eq-introduction-brw-hitting} requires only a linear starting point
condition. For the induction on the number of branching random walks,
this is no longer possible. Indeed, a wave at level $i$ has diameter
of order $r_i$, while its distance from $z$ is also of order $r_i$.
On the other hand, the induction hypothesis for the remaining
$p-1$ branching random walks requires the nonlinear starting point
condition
\begin{equation*}
    \operatorname{dist}(z,K)
    \geq
    C_{p-1}\operatorname{diam}(K)^{\beta_{p-1}}
\end{equation*}
for the target set $K$, where the exponents $\beta_p$ are defined recursively in~\eqref{eq-chi-beta-several-brw} below.

To overcome this obstruction, we refine each annular subtree into
local branching random walk blocks of radius
\begin{equation}
\label{eq-definition-local-scale-general}
    \ell_i
    :=
    \left\lfloor
        c_\star r_i^{1/\beta_{p-1}}
    \right\rfloor,
\end{equation}
where $c_\star>0$ is chosen sufficiently small. Each block then has
diameter of order $\ell_i$, while its distance from $z$ remains
comparable to $r_i$. Consequently, $
    r_i
    \gtrsim
    \ell_i^{\beta_{p-1}}$,
which is precisely the scale relation needed to verify the starting
point condition in the induction hypothesis.

This local block decomposition is the only new geometric ingredient
required to pass from two to an arbitrary number of critical
branching random walks.

Fix $i\geq-I$ and $k\geq0$. Recall that $\eta_i^k$ is the $k$-th
entrance line on $\Sigma_i$, and that $
    \mathcal T_c^{1,\eta_i^k}(\mathscr C_i)$
is the contribution of the descendants of the particles in
$\eta_i^k$ before they first reach $\partial\mathscr C_i$.
For every $u\in\eta_i^k$, set $
    x_u:=V_c^1(u)$.
Since $x_u\in\Sigma_i$ and the annulus $\mathscr C_i$ has diameter
comparable to $r_i$, we have
\begin{equation*}
    \mathcal T_c^{1,u}(\mathscr C_i)
    \subseteq
    \mathcal T_c^{1,u}
    \left(
        B(x_u,3r_i)
    \right).
\end{equation*}

For a vertex $v\succeq u$, define the first-exit stopping line
\begin{equation*}
\label{eq-local-exit-line-general}
\begin{split}
\eta_{\ell_i}(v)
:=
\bigl\{
    w\succeq v:\,
    \|V_c^1(w)-V_c^1(v)\|>\ell_i,
    \text{ and }
    \|V_c^1(y)-V_c^1(v)\|\leq\ell_i
    \text{ for every }
    v\preceq y\prec w
\bigr\}.
\end{split}
\end{equation*}
We now define recursively the local generations associated with the
particle $u$. Set $
    D_{i,k,0}(u):=\{u\}$,
and, for every $q\geq0$, let
\begin{equation*}
\label{eq-local-generation-recursion-general}
\begin{split}
D_{i,k,q+1}(u)
:=
\bigcup_{v\in D_{i,k,q}(u)}
\bigl\{
    w\in\eta_{\ell_i}(v):\,
    V_c^1(y)\in\mathscr C_i \forall y:
    v\preceq y\preceq w
\bigr\}.
\end{split}
\end{equation*}
\begin{figure}[tb]
         \centering
         \includegraphics[width=0.5\linewidth]{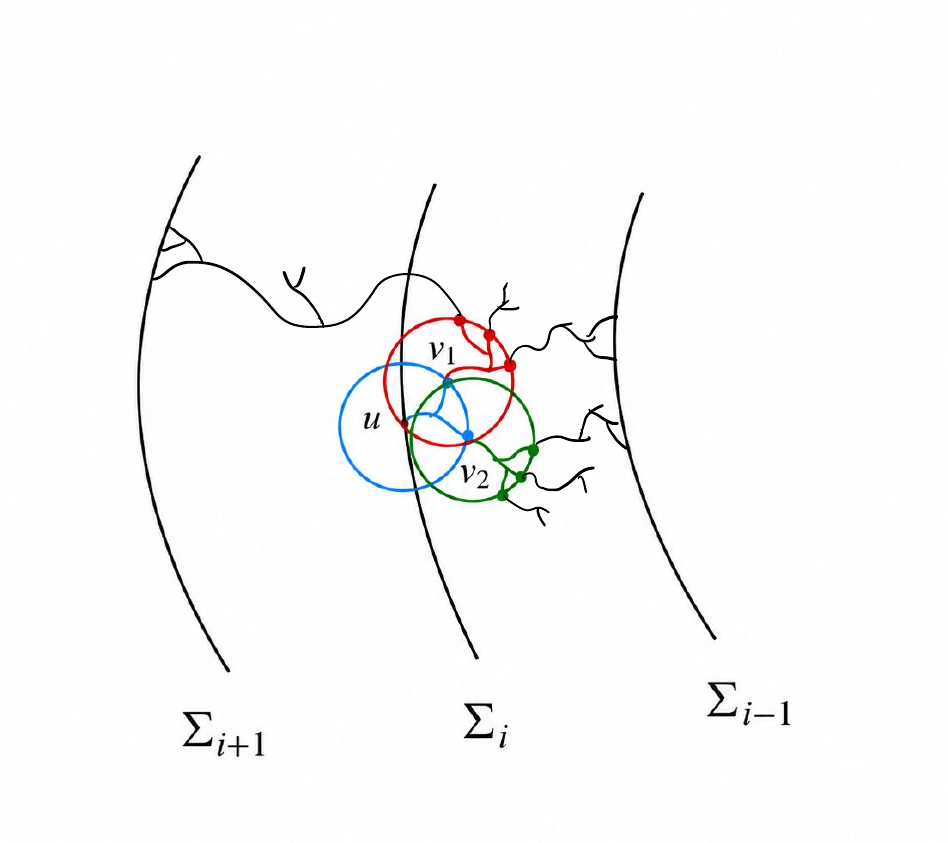}
         \caption{$v_1$ and $v_2$ form $D_{i,k,1}$, and the red and green vertices form $D_{i,k,2}$.}
         \label{refinement}
     \end{figure}
Thus, a vertex of $D_{i,k,q+1}(u)$ is the first descendant of a
vertex of $D_{i,k,q}(u)$ that leaves the ball of radius $\ell_i$
centered at the latter vertex, provided that the ancestral line has
not already left the macroscopic annulus $\mathscr C_i$.

For $v\in D_{i,k,q}(u)$, define the corresponding local spatial block
by $
\mathcal{T}^{1,v}_c(B_{\ell_i})$.
We can see that
\begin{equation*}
\begin{split}
\mathcal T_c^{1,u}(\mathscr C_i)
\subseteq
\bigcup_{q\geq0}
\bigcup_{v\in D_{i,k,q}(u)}
\mathcal{T}^{1,v}_c(B_{\ell_i}).
\end{split}
\end{equation*}
Conditionally on the stopping line $D_{i,k,q}(u)$, the descendant
branching random walks rooted at its vertices are independent.

\begin{lemma}
\label{lemma-number-local-branching-blocks-general}
There exists a constant $C>0$ such that, for every $i\geq-I$, every
$k\geq0$, and every $u\in\eta_i^k$,
\begin{equation}
\label{eq-number-local-branching-blocks-general}
\mathbb E\bigg[
    \sum_{q\geq0}
    \#D_{i,k,q}(u)
    \,\bigg|\,
    \mathcal F_{\eta_i^k}
\bigg]
\leq
C\left(
    1+\frac{r_i^2}{\ell_i^2}
\right).
\end{equation}
\end{lemma}

\begin{proof}
Conditionally on $\mathcal F_{\eta_i^k}$, translate the spatial
positions so that $x_u=0$. Let $(S_n)_{n\geq0}$ be a simple random
walk started from the origin. Define $
    \tau_0:=0$
and, recursively,
\begin{equation*}
    \tau_{q+1}
    :=
    \inf\left\{
        n>\tau_q:
        \|S_n-S_{\tau_q}\|>\ell_i
    \right\},
    \qquad q\geq0.
\end{equation*}
Let
\begin{equation*}
    \sigma_i
    :=
    \inf\left\{
        n\geq0:
        S_n\notin B_{3r_i}
    \right\}.
\end{equation*}

Because $\mathscr C_i\subseteq B(x_u,3r_i)$, the many-to-one formula,
applied successively at the local stopping lines, gives
\begin{equation}
\label{eq-many-to-one-local-generation-general}
\begin{split}
\mathbb E\left[
    \#D_{i,k,q}(u)
    \,\middle|\,
    \mathcal F_{\eta_i^k}
\right]
&\leq
\mathbf P_0\left(
    \tau_q<\sigma_i
\right).
\end{split}
\end{equation}

Define the number of $\ell_i$-pieces before time $n$ by
\begin{equation*}
    \mathcal N_n^{(\ell_i)}
    :=
    \max\left\{q\geq0:\tau_q\leq n
    \right\}.
\end{equation*}
Summing \eqref{eq-many-to-one-local-generation-general} over $q$
gives
\begin{equation}
\label{eq-total-local-generations-general}
\begin{split}
\mathbb E\bigg[
    \sum_{q\geq0}
    \#D_{i,k,q}(u)
    \,\bigg|\,
    \mathcal F_{\eta_i^k}
\bigg]
&\leq
\sum_{q\geq0}
\mathbf P_0\left(
    \tau_q<\sigma_i
\right)\leq
\mathbb E_0\left[
    \mathcal N_{\sigma_i}^{(\ell_i)}+1
\right].
\end{split}
\end{equation}

We now use the second-moment estimate for the number of stopped pieces
of a simple random walk established in
Remark~\ref{remark-second-moment-stopped-pieces}. For every $n\geq1$,
\begin{equation*}
\mathbb E_0\left[
    \left(
        \mathcal N_n^{(\ell_i)}+1
    \right)^2
\right]^{1/2}
\lesssim
1+\frac{n}{\ell_i^2};
\end{equation*}
see Remark \ref{remark-second-moment-stopped-pieces}.
Moreover, the exit time from $B_{3r_i}$ satisfies
\begin{equation*}
    \mathbf P_0\left(
        \sigma_i\geq jr_i^2
    \right)
    \lesssim
    e^{-cj},
    \qquad j\geq0.
\end{equation*}

Decomposing according to the value of $\sigma_i$, using the
monotonicity of $\mathcal N_n^{(\ell_i)}$, and applying the
Cauchy--Schwarz inequality, we obtain
\begin{equation*}
\begin{split}
\mathbb E_0\left[
    \mathcal N_{\sigma_i}^{(\ell_i)}+1
\right]
&=
\sum_{j\geq0}
\mathbb E_0\left[
    \left(
        \mathcal N_{\sigma_i}^{(\ell_i)}+1
    \right)
    \mathds{1}_{\{jr_i^2\leq\sigma_i\leq (j+1)r_i^2\}}
\right]\\
&
\leq
\sum_{j\geq0}
\mathbb E_0\left[
    \left(
        \mathcal N_{(j+1)r_i^2}^{(\ell_i)}+1
    \right)^2
\right]^{1/2}
\mathbf P_0\left(
    \sigma_i\geq jr_i^2
\right)^{1/2}
\\
&\lesssim
\sum_{j\geq0}
\left(
    1+
    \frac{(j+1)r_i^2}{\ell_i^2}
\right)e^{-cj/2}
\lesssim
1+\frac{r_i^2}{\ell_i^2}.
\end{split}
\end{equation*}
Together with \eqref{eq-total-local-generations-general}, this proves
the lemma.
\end{proof}

\subsection{Completion of the upper bound in Proposition~\ref{prop-hitting-several-brw}}

Let us define the exponents $\beta_p$ recursively as follows:
\begin{equation}
\label{eq-chi-beta-several-brw}
    \beta_1:=1,
    \qquad
    \beta_p:=2+\chi_p\beta_{p-1},
    \qquad
    \chi_p:=2p-1+\sum_{j=1}^{2p-2}\frac1j,
    \quad p\geq2 .
\end{equation}
Notice that $\chi_p=a_{2p-2}$ in Equation~\eqref{eq-a-m}.
\begin{proof}[Upper bound in Proposition~\ref{prop-hitting-several-brw}]
The case $p=2$ was proved in Section~\ref{section8}. Assume the result for $p-1$ and
choose $I$ so that $
    r_{-I}
    \asymp
    \operatorname{diam}(A)^{\chi_p\beta_{p-1}}$.
The choice \eqref{eq-definition-local-scale-general} ensures both $\ell_i\gtrsim \operatorname{diam}(A)^{\chi_p}$
and, for every local block root $v$,
\begin{equation*}
    \|z-V_c^1(v)\|
    \asymp r_i
    \geq
    C_{p-1}\left(c_1\ell_i+\operatorname{diam}(A)\right)^{\beta_{p-1}}.
\end{equation*}
Thus, since $\ell_i \gtrsim \text{diam}(A)^{\chi_p}=\text{diam}(A)^{a_{2p-2}}$, we can apply Equation~\eqref{eq-introduction-stopped-brw-capacity} of Theorem~\ref{theorem-main-capacity-introduction} for $N=2p$ and get 
\begin{equation*}
\mathbb E\left[
    \operatorname{Cap}_{d-4(p-1)}
    \left(\mathcal T^{1,v}_c(B_{\ell_i})+A\right)
\right]
\lesssim
\ell_i^2\operatorname{Cap}_{d-4p}(A).
\end{equation*}
Combining this estimate with
Lemma~\ref{lemma-number-local-branching-blocks-general} gives, for one
entrance particle $u\in\eta_i^k$,
\begin{equation}
\label{eq-one-annular-subtree-several-brw}
\begin{split}
&\mathbb P\left(z+\mathcal T_c^2+\cdots+\mathcal T_c^p\conn\mathcal T_c^{1,u}(\mathscr C_i)+A
    \;\middle|\;\mathcal F_{\eta_i^k}
\right)
\\
&\qquad \qquad\lesssim
\Big(1+\frac{r_i^2}{\ell_i^2}\Big)
\frac{\ell_i^2\operatorname{Cap}_{d-4p}(A)}
{r_i^{d-2(p-1)}}
\lesssim
r_i^{2p-d}\operatorname{Cap}_{d-4p}(A).
\end{split}
\end{equation}
We denote by
\begin{equation*}
\begin{split}
P_{0}^{(p)}
& :=
\mathbb{P}\left(
        z+\mathcal{T}_c^2+\cdots+\mathcal{T}_c^p\conn
            \mathcal{T}_c^1\cap B(z,r_{-I})+A
\right), \\ 
P_i^{(p)}
& :=
\mathbb{P}\left(
        z+\mathcal{T}_c^2+\cdots+\mathcal{T}_c^p \conn
        \mathcal{T}_c^{(1,i)}+A
\right).
\end{split}
\end{equation*}
After summing over the entrance lines in Equation~\eqref{eq-one-annular-subtree-several-brw} and using
\eqref{total-number-waves},
\begin{equation*}
P_i^{(p)}
\lesssim
r_i^{2p-d}\operatorname{Cap}_{d-4p}(A)
\left(
    \frac{r_i^{d-2}}{\|z\|^{d-2}}\mathds 1_{\{-I\leq i<0\}}
    +\mathds 1_{\{i\geq0\}}
\right).
\end{equation*}
Summing over $i \geq -I$, we get
\begin{equation}
\label{eq-annular-sum-several-brw}
    \sum_{i\geq-I}P_i^{(p)}
    \lesssim
    \frac{\operatorname{Cap}_{d-4p}(A)}{\|z\|^{d-2p}}.
\end{equation}

The inner argument is the same as for two branching random walks.
The only new inputs are
\begin{equation*}
\mathbb E\left[
    \left|\big(\mathcal T_c^2+\cdots+\mathcal T_c^p\big)
    \cap B(0,2r_{-I})\right|
\right]
\lesssim
r_{-I}^{2(p-1)}
\end{equation*}
and
\begin{equation*}
    \operatorname{Cap}_{d-4}(A)
    \lesssim
    \operatorname{diam}(A)^{4(p-1)}\operatorname{Cap}_{d-4p}(A).
\end{equation*}
Consequently,
\begin{equation*}
P_{0}^{(p)}
\lesssim
\frac{\operatorname{diam}(A)^{4(p-1)}r_{-I}^{2(p-1)}}{\|z\|^{d-2}}
\operatorname{Cap}_{d-4p}(A).
\end{equation*}
Since $r_{-I} \asymp \chi_p\, \beta_{p-1} $ and $\|z\| \gtrsim \text{diam}(A)^{\beta_p}$, then under the condition that $\beta_p=2+\chi_p\,\beta_{p-1}$, the last display
is bounded by $\frac{
    \operatorname{Cap}_{d-4p}(A)
}{
    \|z\|^{d-2p}
}$.
Together with \eqref{eq-annular-sum-several-brw}, this completes the
induction.
\end{proof}

\subsection{Mixed sums of random walk and branching random walk ranges}
\label{section-general-mixed-hitting}

We are now ready to prove Theorem \ref{proposition-main-hitting-introduction}.
We only highlight the main steps of the proof since it follows the same lines as in the previous subsection.

For the lower bound, we localize all the random walk ranges inside a ball and apply
Proposition~\ref{prop-hitting-several-brw}, followed by Proposition~\ref{prop-capacity-several-stopped-rw}. 

For any $R >0$, we have
\begin{equation*}
    \begin{split}
        \mathbb P\bigg(
        z+\sum_{i=1}^{m}\mathcal S^i
        +\sum_{j=1}^{k}\mathcal T_c^j
    \conn A
\bigg)& \geq \mathbb P\bigg(
        z
        +\sum_{j=1}^{k}\mathcal T_c^j
    \conn \sum_{i=1}^{m}\mathcal S^i(B_R)+A
\bigg)
    \end{split}
\end{equation*}
Choosing $R=\lfloor c \|z\|\rfloor$ with c sufficiently small so that we can use Corollaries \ref{prop-hitting-several-brw} and \ref{prop-capacity-several-stopped-rw} and get
\begin{equation*}
    \begin{split}
        \mathbb P\bigg(
        z
        +\sum_{j=1}^{k}\mathcal T_c^j
    \conn \sum_{i=1}^{m}\mathcal S^i(B_R)+A
\bigg) & \gtrsim\frac{\mathrm{Cap}_{d-4k-2m}(A)}{\|z\|^{d-2m-2k}}
    \end{split}
\end{equation*}

For the upper bound, we proceed by induction on $m$ and use the annular
decomposition of a random walk as done in Section \ref{subsubsection-one-rw-one-brw}.
We use the following exponents $\alpha_{m,k}$ defined by induction as follows:
and, for $m\geq0$ and $k\geq1$,
\begin{equation}\label{eq-alpha}
    \alpha_{0,k}:=\beta_k, \qquad     \alpha_{m,k}
    :=
    \alpha_{m-1,k}
    +
    \frac{m+2k-1}{m+k-1}.
\end{equation}
The case $m=0$ is precisely
Proposition~\ref{prop-hitting-several-brw}, which holds under the
condition $\|z\| \geq C_k \text{diam}(A)^{\beta_k}$, and therefore $\alpha_{0,k}=\beta_k$.

For the induction step, we choose the smallest scale in the annular
decomposition so that $r_{-I}\asymp\operatorname{diam}(A)^{\alpha_{m-1,k}}$, which allows us to apply the induction hypothesis to the remaining
$m-1$ random walk ranges.
For the inner contribution, arguing as above, we obtain
\begin{equation*}
    \begin{split}
        \mathbb{P}\bigg(z+\sum_{i=2}^{m}\mathcal{S}^i+\sum_{j=1}^k \mathcal{T}_c^j & \conn \mathcal{S}^1\cap B(z,r_{-I})+A \bigg)\\ & \lesssim \frac{\mathrm{Cap}_{d-2}(A)}{\|z\|^{d-2}}\mathbb{E}\bigg[\bigg|\Big(z+\sum_{i=2}^{m}\mathcal{S}^i+\sum_{j=1}^k\mathcal{T}_c^j\Big)\cap B(z,r_{-I})\bigg|\bigg] \\
        & \lesssim \mathrm{Cap}_{d-2m-4k}(A)\frac{\text{diam}(A)^{2m+4k-2}\,r_{-I}^{2(m-1+k)}}{\|z\|^{d-2}}  
    \end{split}
\end{equation*}
The last term is comparable to $\frac{\mathrm{Cap}_{d-2m-4k}(A)}{\|z\|^{d-2m-2k}}$ thanks to the relation \eqref{eq-alpha};
this explains the choice of the exponent $\alpha_{m,k}$.
\appendix

\section{Some Technical estimates}

We present here the proofs of some estimates that were used in the paper.

\subsection{On properties of the capacity: proof of Proposition~\ref{prop-basic-capacity-properties}}
\label{appendix-capacities}

Recall the definition~\eqref{eq-gamma-capacity-definition} of the $\gamma$-capacity.

\begin{proof}
\textup{(i)} The monotonicity \eqref{eq-capacity-monotonicity} follows immediately from
the definition of $\gamma$-capacity, since, whenever $A\subset B$, every
probability measure supported on $A$ is also a probability measure
supported on $B$.

\textup{(ii)}
For the subadditivity property \eqref{eq-capacity-subadditivity}, we refer
to \cite[Claim~2.1]{asselah2024branching}.

\textup{(iii)}
The identities in \eqref{eq-capacity-invariances} follow by translating
or reflecting a probability measure.

\textup{(iv)}
Since $A+B=\bigcup_{b\in B}(A+b)$, subadditivity and translation invariance imply \eqref{capofsum}.

\textup{(v)}
If $x,y\in A$, then we have
\begin{equation*}
    \frac{1}{(1+\|x-y\|)^\alpha}
    \geq
    \frac{1}{(1+\operatorname{diam}(A))^{\alpha-\beta}}
    \frac{1}{(1+\|x-y\|)^\beta}.
\end{equation*}
Taking energies $\mathcal{E}_{k_{\gamma}}(\nu)$, then the infima over probability measures $\nu$ on $A$, and
finally the reciprocals yields the inequality \eqref{comparisonof2cap}.
\end{proof}

\subsection{On random walks estimates: proof of Lemma~\ref{lemma-number-stopped-pieces}}\label{appendix-rw}

\begin{proof}
For every $i\geq 1$, set $
    \sigma_i
    :=
    \tau_i-\tau_{i-1}$.
By the strong Markov property, the random variables
$(\sigma_i)_{i\geq 1}$ are independent and identically distributed,
with the same law as
\begin{equation*}
    H_R
    :=
    \inf\left\{
        m\geq 1:
        \|S_m\|>R
    \right\}.
\end{equation*}

We first observe that there exist constants $c_0,p_0>0$, depending
only on the dimension, such that, uniformly in $R\geq 1$
\begin{equation}
\label{exit-time-lower-tail}
    \mathbb{P}\left(
        H_R\geq c_0R^2
    \right)
    \geq p_0.
\end{equation} Indeed, by Doob's maximal inequality applied
to the nonnegative submartingale
$(\|S_m\|^2)_{m\geq 0}$,
\begin{equation*}
\begin{split}
\mathbb{P}\left(
    H_R<c_0R^2
\right)
&\leq
\mathbb{P}\left(
    \max_{0\leq m\leq \lfloor c_0R^2\rfloor}
    \|S_m\|>R
\right)
\leq
\frac{
    \mathbb{E}\left[
        \|S_{\lfloor c_0R^2\rfloor}\|^2
    \right]
}{
    R^2
}
\leq
c_0.
\end{split}
\end{equation*}
So we can take $p_0=1-c_0$.
It follows from \eqref{exit-time-lower-tail} that
\begin{equation*}
\begin{split}
\mathbb{E}\left[
    \exp\left(
        -\frac{H_R}{R^2}
    \right)
\right]
&\leq
1-p_0\left(
    1-e^{-c_0}
\right)=:q,
\end{split}
\end{equation*}
where $q\in(0,1)$ is independent of $R$.

For every $k\geq 1$, the definition of $\mathcal{N}_n$ gives $
    \left\{
        \mathcal{N}_n\geq k
    \right\}
    =
    \left\{
        \tau_k\leq n
    \right\}$.
Therefore, by Markov's inequality,
\begin{equation}\label{remark-appendix}
\begin{split}
\mathbb{P}\left(
    \mathcal{N}_n\geq k
\right)
&=
\mathbb{P}\left(
    \exp\left(
        -\frac{\tau_k}{R^2}
    \right)
    \geq
    \exp\left(
        -\frac{n}{R^2}
    \right)
\right)
\\
&\leq
\exp\left(
    \frac{n}{R^2}
\right)
\mathbb{E}\left[
    \exp\left(
        -\frac{\tau_k}{R^2}
    \right)
\right]
\\
&=
\exp\left(
    \frac{n}{R^2}
\right)
\left(
    \mathbb{E}\left[
        \exp\left(
            -\frac{H_R}{R^2}
        \right)
    \right]
\right)^k
\leq
\exp\left(
    \frac{n}{R^2}
\right)q^k.
\end{split}
\end{equation}

Let $
    k_0
    :=
    \left\lceil
        \frac{2}{|\log q|}
        \left(
            1+\frac{n}{R^2}
        \right)
    \right\rceil$.
Then
\begin{equation*}
\begin{split}
\mathbb{E}\left[
    \mathcal{N}_n+1
\right]
&=
\sum_{k\geq 0}
\mathbb{P}\left(
    \mathcal{N}_n\geq k
\right)
\leq
k_0
+
\exp\left(
    \frac{n}{R^2}
\right)
\sum_{k\geq k_0}q^k
\lesssim
1+\frac{n}{R^2}.
\end{split}
\end{equation*}
This proves the lemma.
\end{proof}
\begin{remark}
\label{remark-second-moment-stopped-pieces}
The exponential-tail estimate used in the proof of
Lemma~\ref{lemma-number-stopped-pieces} also yields
\begin{equation*}
    \mathbb{E}\left[
        \left(
            \mathcal{N}_n+1
        \right)^2
    \right]^{1/2}
    \lesssim
    1+\frac{n}{R^2}.
\end{equation*}
Indeed, Equation~\eqref{remark-appendix} gives constants $q\in(0,1)$ such that $
    \mathbb{P}\left(
        \mathcal{N}_n\geq k
    \right)
    \leq
    \exp\left(
        \frac{n}{R^2}
    \right)q^k.$
The result then follows by writing $
    \mathbb{E}\left[
        \left(
            \mathcal{N}_n+1
        \right)^2
    \right]
    =
    \sum_{k\geq0}
    (2k+1)
    \mathbb{P}\left(
        \mathcal{N}_n\geq k
    \right)$
and splitting the sum at an index of order $n/R^2$.
\end{remark}

\paragraph{Acknowledgment.} The author would like to thank Amine Asselah and Yueyun Hu for helpful discussions and guidance. Many thanks to Tianyi Bai for explaining us the approach of \cite{baiperes}. Finally, it is a pleasure to thank Quentin Berger for having read the entire manuscript, for his advice which improved the presentation, and for his warm encouragements.
\paragraph{AI disclosure.}No AI assistance was used in this work.
\printbibliography
\end{document}

%% file: prefix.tex
\def\UseBibLatex{1}

\makeatletter
\def\input@path{{styles/}}
\makeatother

\newcommand{\UsePackage}[1]{%
  \IfFileExists{styles/#1.sty}{%
      \usepackage{styles/#1}%
   }{%
      \IfFileExists{../styles/#1.sty}{%
         \usepackage{../styles/#1}%
      }{%
         \usepackage{#1}%
      }%
   }%
}

\usepackage[T1]{fontenc}
\usepackage{lmodern}
\usepackage{textcomp}

\usepackage{amsmath}%
\usepackage{amssymb}%
\usepackage[table]{xcolor}%

\usepackage{todonotes}
\usepackage[in]{fullpage}%

\usepackage[amsmath,thmmarks]{ntheorem}%
\theoremseparator{.}%

\usepackage{titlesec}%
\titlelabel{\thetitle. }%
\usepackage{xcolor}%
\usepackage{mleftright}%
\usepackage{xspace}%
\usepackage{graphicx}
\usepackage{hyperref}%
\usepackage{mathtools}%
\usepackage{enumitem}

\usepackage{hyperref}%
\hypersetup{%
      unicode,
      breaklinks,%
      colorlinks=true,%
      urlcolor=[rgb]{0.25,0.0,0.0},%
      linkcolor=[rgb]{0.5,0.0,0.0},%
      citecolor=[rgb]{0,0.2,0.445},%
      filecolor=[rgb]{0,0,0.4},
      anchorcolor=[rgb]={0.0,0.1,0.2}%
}
\usepackage[ocgcolorlinks]{ocgx2}

\newtheorem{theorem}{Theorem}[section]

\newtheorem{lemma}[theorem]{Lemma}
\newtheorem{corollary}[theorem]{Corollary}
\newtheorem{proposition}[theorem]{Proposition}

\newtheorem{remark}[theorem]{Remark}%

\newcommand{\myqedsymbol}{\rule{2mm}{2mm}}

\theoremheaderfont{\em}%
\theorembodyfont{\upshape}%
\theoremstyle{nonumberplain}%
\theoremseparator{}%
\theoremsymbol{\myqedsymbol}%

\newtheorem{proof}{Proof:}%

\providecommand{\emphind}[1]{}%
\renewcommand{\emphind}[1]{\emph{#1}\index{#1}}

\definecolor{blue25emph}{rgb}{0, 0, 11}

\providecommand{\emphic}[2]{}
\renewcommand{\emphic}[2]{\textcolor{blue25emph}{%
      \textbf{\emph{#1}}}\index{#2}}

\providecommand{\emphi}[1]{}%
\renewcommand{\emphi}[1]{\emphic{#1}{#1}}

\definecolor{almostblack}{rgb}{0, 0, 0.3}

\providecommand{\emphw}[1]{}%
\renewcommand{\emphw}[1]{{\textcolor{almostblack}{\emph{#1}}}}%

\providecommand{\emphOnly}[1]{}%
\renewcommand{\emphOnly}[1]{\emph{\textcolor{blue25}{\textbf{#1}}}}

\newcommand{\remove}[1]{}%

\providecommand{\BibLatexMode}[1]{}
\providecommand{\BibTexMode}[1]{}

\ifx\UseBibLatex\undefined%
  \renewcommand{\BibLatexMode}[1]{}
  \renewcommand{\BibTexMode}[1]{#1}
\else
  \renewcommand{\BibLatexMode}[1]{#1}
  \renewcommand{\BibTexMode}[1]{}
\fi

\BibLatexMode{%
   \usepackage[bibencoding=utf8,style=alphabetic,backend=biber]{biblatex}%
   \UsePackage{my_biblatex}%
}

\numberwithin{figure}{section}%
\numberwithin{table}{section}%
\numberwithin{equation}{section}%

%% file: template.bib
@article{Zhu1,
  author  = {Zhu, Qingsan},
  title   = {On the critical branching random walk I: branching capacity and visiting probability},
  journal = {Probability Theory and Related Fields},
  year    = {2026},
  doi     = {10.1007/s00440-026-01530-9},
  url     = {https://doi.org/10.1007/s00440-026-01530-9},
  issn    = {1432-2064}
}

@article{asselah2023localtimescapacitytransient,
  author  = {Asselah, Amine and Schapira, Bruno and Sousi, Perla},
  title   = {Local times and capacity for transient branching random walks},
  journal = {Probability Theory and Related Fields},
  year    = {2026},
  volume  = {194},
  number  = {3},
  pages   = {1443--1493},
  doi     = {10.1007/s00440-025-01378-5},
  url     = {https://doi.org/10.1007/s00440-025-01378-5},
  issn    = {1432-2064}
}

@book{modernintro,
author = {Lawler, Gregory F. and Limic, Vlada},
title = {Random Walk: A Modern Introduction},
year = {2010},
isbn = {0521519187},
publisher = {Cambridge University Press},
address = {USA}
}

@article{FS1989,
author = {Fitzsimmons, P. J. and Salisbury, Thomas S.},
journal = {Annales de l'I.H.P. Probabilités et statistiques},
language = {eng},
number = {3},
pages = {325-350},
publisher = {Gauthier-Villars},
title = {Capacity and energy for multiparameter Markov processes},
url = {http://eudml.org/doc/77355},
volume = {25},
year = {1989},
}

@misc{baiperes,
      title={Minkowski sum of fractal percolation and random sets}, 
      author={Tianyi Bai and Xinxin Chen and Yuval Peres},
      year={2024},
      eprint={2412.16415},
      archivePrefix={arXiv},
      primaryClass={math.PR},
      url={https://arxiv.org/abs/2412.16415}, 
}

@book{ItoMcKean1974,
  author    = {It{\^o}, K. and McKean, H.},
  title     = {Diffusion Processes and Their Sample Paths},
  series    = {Grundlehren der mathematischen Wissenschaften},
  volume    = {125},
  publisher = {Springer-Verlag},
  address   = {Berlin},
  year      = {1974}
}

@incollection{Spitzer,
  author    = {Spitzer, Frank},
  title     = {Discussion on Professor Kingman's Paper},
  booktitle = {The Annals of Probability},
  volume    = {1},
  number    = {6},
  pages     = {900--909},
  year      = {1973},
  publisher = {Institute of Mathematical Statistics},
  doi       = {10.1214/aop/1176996799},
  url       = {https://projecteuclid.org/journals/annals-of-probability/volume-1/issue-6/Discussion-on-Professor-Kingmans-Paper/10.1214/aop/1176996799.full}
}

@book{grimmett-stirzaker,
  author    = {Grimmett, Geoffrey R. and Stirzaker, David R.},
  title     = {Probability and Random Processes},
  publisher = {Oxford University Press},
  edition   = {3},
  year      = {2001}
}

@article{asselah2024branching,
  author  = {Asselah, Amine and Okada, Izumi and Schapira, Bruno and Sousi, Perla},
  title   = {Branching random walks and Minkowski sum of random walks},
  journal = {Probability Theory and Related Fields},
  year    = {2025},
  volume  = {191},
  number  = {3-4},
  pages   = {1289--1322},
  doi     = {10.1007/s00440-024-01352-7}
}
